\RequirePackage{fix-cm}
 \documentclass[smallextended]{svjour3} 
  \def\makeheadbox{\relax}
\usepackage{amssymb,amsmath,amsfonts,latexsym} %%
\usepackage{hyperref}

\usepackage{mathtools}
\usepackage{multirow,array}
\usepackage{graphicx}
\usepackage{subfigure}
\usepackage{epstopdf}
\usepackage{float} 
\usepackage{color, xcolor}
\allowdisplaybreaks
\usepackage{soul}
\usepackage{lipsum}
\usepackage{amsfonts}
\usepackage{graphicx}
\usepackage{epstopdf}
\usepackage{pifont}
\usepackage{calligra}
\usepackage{hyperref}
\usepackage{autonum}
\usepackage{hhline}
\usepackage{array}
\usepackage{diagbox}
\usepackage{tcolorbox}
\usepackage{mdframed}
\usepackage{multicol}
\usepackage{graphicx}
\usepackage{moreverb}
\usepackage{bbm}
\usepackage[margin=1.34in]{geometry}
\usepackage{todonotes}
\usepackage{scalerel,amssymb}
\allowdisplaybreaks
\usepackage{mathrsfs}  
\usepackage{lineno}
\usepackage{todonotes}
\usepackage{tikz}
\usepackage{appendix}
\usepackage{pgfplots}
\usetikzlibrary{arrows.meta}
\usepackage[numbers,sort&compress]{natbib}
\definecolor{mygreen}{HTML}{43a047}
\usepackage{doi}
\usepackage{alphalph}
\usepackage{booktabs}
\usepackage{makecell}
\usepackage{adjustbox}
\usepackage{array}
\newcolumntype{H}{>{\setbox0=\hbox\bgroup}c<{\egroup}@{}}

\newcommand{\Om}{\Omega}

\DeclareMathOperator*{\esssup}{ess\,sup}

\def\aaa{\mathfrak{a}}
\def\bbb{\mathfrak{b}}
\newcommand{\psitwo}{\psi_2^{\frakKone}}

\def \xin{\xi^{(n)}}
\def \xit{\xi^{(n)}_t}
\def \xitt{\xi^{(n)}_{tt}}

\def \bxin{\boldsymbol{\xi}}
\def \bxit{\boldsymbol{\xi_t}}
\def \bxitt{\boldsymbol{\xi_{tt}}}

\def \psit{\psi_t}
\def \psitt{\psi_{tt}}

\def \taua {\tau^a}

\newcommand{\bxi}{\boldsymbol{\xi}}
\newcommand{\ds}{\, \textup{d} s }
\newcommand{\dx}{\, \textup{d} x}

\newcommand{\dxs}{\, \textup{d}x\textup{d}s}

\newcommand{\intTO}{\int_0^T \int_{\Omega}}

\newcommand{\intt}{\int_0^t}
\newcommand{\intT}{\int_0^T}
\newcommand{\intO}{\int_{\Omega}}
\newcommand{\R}{\mathbb{R}} 
\newcommand{\N}{\mathbb{N}} 
\newcommand{\Honezero}{H_0^1(\Omega)}

\newcommand{\bfq}{\boldsymbol{q}}

\newcommand{\frakKone}{\mathfrak{K}_1}
\newcommand{\tfrakKone}{\tilde{\mathfrak{K}}_1}
\newcommand{\frakKtwo}{\mathfrak{K}_2}
\newcommand\Lconv{\ast}

\newcommand{\rt}{\tau_\theta}

\newcommand{\frakR}{\mathfrak{r}}
\newcommand{\mm}{\mathcal M}
\definecolor{grey}{rgb}{0.5,0.5,0.5}
 
\usepackage{stackengine,scalerel}
\newcommand\dhookrightarrow{\mathrel{\ThisStyle{\abovebaseline[-.6\LMex]{%
				\ensurestackMath{\stackanchor[.15\LMex]{\SavedStyle\hookrightarrow}{%
						\SavedStyle\hookrightarrow}}}}}}

\colorlet{brown}{brown!80!black}

\definecolor{darkgreen}{rgb}{0,0.5,0}

\newcommand{\genk}{\mathfrak{K}}
\newcommand{\calX}{\mathcal{X}}
\makeatletter
\newcommand{\leqnomode}{\tagsleft@true\let\veqno\@@leqno}
\newcommand{\reqnomode}{\tagsleft@false\let\veqno\@@eqno}
\makeatother
\begin{document}

\title{A unified analysis framework for generalized fractional Moore--Gibson--Thompson equations: Well-posedness and singular limits}

\titlerunning{A unified analysis framework for generalized fractional MGT equations}
\author{
        Mostafa Meliani }

\authorrunning{M. Meliani}
\institute{Mostafa Meliani
\at
Department of Mathematics, Radboud university,\\ Heyendaalseweg 135,\\
6525 AJ Nijmegen, The Netherlands \\
\email{mostafa.meliani@ru.nl} 
}

\date{}

\maketitle
\begin{abstract}
In acoustics, higher-order-in-time equations arise when taking into account a class of thermal relaxation laws in the modeling of sound wave propagation. In this work, we analyze initial boundary value problems for a family of such equations and determine the behavior of solutions as the relaxation time vanishes. In particular, we allow the leading term to be of fractional type. The studied model can be viewed as a generalization of the well-established (fractional) Moore--Gibson--Thompson equation with three, in general nonlocal, convolution terms involving two different kernels. The interplay of these convolutions will influence the uniform analysis and the limiting procedure. To unify the theoretical treatment of this class of local and nonlocal higher-order equations, we relax the classical assumption on the leading-term kernel and consider it to be a Radon measure. After establishing uniform well-posedness with respect to the relaxation time of the considered general model, we connect it, through a delicate singular limit procedure, to fractional second-order models of linear acoustics. 
 
\end{abstract}
\section{Introduction} \label{Sec:Introduction}
In acoustics, higher-order-in-time equations arise when modeling heat exchanges in the medium through a thermally relaxed flux law. In this work, we investigate such an equation of the form:
\begin{equation}\label{eqn:first_eq}
	\tau^a\big(\frakKone \Lconv \psi_{tt})_{t} + \aaa\psi_{tt} - c^2  \tau^a\frakKone \Lconv \Delta\psi_t - c^2 \bbb \Delta \psi -   \delta \rt^b \frakKtwo \Lconv \Delta\psi_{tt}   = f.
\end{equation} 
Equation~\eqref{eqn:first_eq} is a generalization of the widely studied Moore--Gibson--Thompson (MGT) equation of linear acoustics:
\begin{equation}\label{eq:MGT}
	\tau\psi_{ttt} + \psi_{tt} - c^2  \tau\Delta\psi_t - c^2 \Delta \psi -   \delta \Delta\psi_{t}   = f;
\end{equation}
see, for example, \cite{kaltenbacher2011wellposedness,lasiecka2015moore,dell2017moore,pellicer2019wellposedness,bucci2020regularity} and the references contained therein for mathematical studies of the latter. 
The MGT equation is then obtained by formally setting $\aaa=\bbb=1$, $\frakKone = \delta_0$, $\frakKtwo =1$, $a=1$, and $b=0$ in \eqref{eqn:first_eq}, where $\delta_0$ is the Dirac delta distribution.
Our goal is to provide the analysis of equations \eqref{eqn:first_eq} and \eqref{eq:MGT} in a unified framework; that is, we view both models through the lens of equation \eqref{eqn:first_eq} with kernels allowed to be Radon measures.\\
\indent Equation \eqref{eqn:first_eq} comes about when assuming heat flux laws of more general type than Maxwell--Cattaneo's~\cite{cattaneo1958forme}. These are, among others, relevant in complex heterogeneous media that exhibit anomalous diffusion. 
We refer the reader to~\cite{povstenko2015fractional,compte1997generalized,zhang2014time} for discussions on general heat flux laws. The derivation of \eqref{eqn:first_eq} with Caputo--Dzhrbashyan derivatives is due to~\cite{kaltenbacher2022time}; the justification of the generalized fractional equation in Section~\ref{Sec:Modeling} will follow a similar reasoning. \\
\indent Note that the leading term $(\tau^a\frakKone \Lconv \psi_{tt})_{t}$ will necessitate prescribing the third initial condition on $(\frakKone\Lconv\psitt)(0)$. This is due to the term having the form of a generalized Riemann--Liouville derivative on $\psitt$. We refer the reader to \cite[Chapter 2]{podlubny1998fractional} for a discussion on the appropriate initial conditions for different fractional derivatives. This is obviously different from other works on time-fractional MGT equations with exclusively Caputo--Dzhrbashyan derivatives such as the aforementioned~\cite{kaltenbacher2022time}. Nevertheless, we should mention that the results here cover the linear fractional MGT equations of~\cite{kaltenbacher2022time} when the third initial condition $\psitt(0) = 0$ is prescribed. A discussion on the choice of the leading term's form is provided in Section~\ref{Sec:Modeling}.
\\
\indent Due to the presence of the convolutions $\mathfrak{K}_{1,2}\Lconv \cdot$, equation~\eqref{eqn:first_eq} is, in general, nonlocal in time. Nonlocal wave equations have been studied by a number of authors in different settings. Traditionally confined to that of a smooth kernel with leading integer order~\cite{alves2018moore,dell2016moore,conti2006singular}, there has been an increasing number of works using singular kernels~\cite{kaltenbacher2022limiting,kaltenbacher2021determining}. The motivation to consider these has originally stemmed from fractional derivative kernels, which are widely studied by the fractional calculus community~\cite{kian2017existence, liu2022uniqueness,jin2021fractional,kaltenbacher2022time}. Here we go one step further by viewing the leading-term kernel $\frakKone$ as a Radon measure. Thus, a Dirac pulse, $\delta_0$, would correspond to a unit point mass measure at 0, while a fractional kernel would be identified with an absolutely continuous measure.
\\
\indent From a mathematical perspective allowing for more general kernels is justified by the fact that fractional derivative kernels are but a subset of a much larger family. 
In fact, Sonine~\cite{sonine1884generalisation} showed that the resolution of Abel's classical mechanical problem~\cite{abel1826resolution} is owed to the fact that the fractional derivative kernel
\(
g_{\frac12}(t):= \dfrac{1}{\Gamma(\frac12)} t^{-\frac12}
\)
has a resolvent, i.e., $\intt g_{\frac12}(t-s) g_{\frac12}(s)\ds = 1$ for all $t \in \R_+$. This perspective not only allows us to solve more general fractional differential equations, it also provides us with a better understanding of the way fractional derivatives behave. In fact, with this perspective, it becomes natural to put emphasis on the resolvent kernel when discussing the well-posedness of \eqref{eqn:first_eq}.
\\
\indent We intend to take the same viewpoint as Sonine, and show that \eqref{eqn:first_eq}, supplemented with appropriate initial and boundary data, is well-posed provided $\frakKone$ has a ``regular enough" resolvent $\tfrakKone$. A discussion on the sufficient regularity of $\tfrakKone$ is given in Section~\ref{Sec:preliminaries}, alongside useful generalizations of well-known results from functional analysis. 
\\
\indent In equation \eqref{eqn:first_eq}, the relaxation time $\tau$ plays an important role in the behavior of the model. Indeed, $\tau$ characterizes the time lag between a temperature change and the ensuing heat flux variation. Understanding how higher-order acoustic equations behave when the relaxation time $\tau$ is sent to $0$ (the limit at which temperature changes are felt immediately by the medium) has been a topic of recent interest in the mathematical acoustic community; see e.g., \cite{bongarti2020singular,bongarti2020vanishing,kaltenbacher2020vanishing}. We mention that vanishing relaxation time limits are also of interest for second-order models with a $\tau$-dependent memory~\cite{conti2005singular,conti2006singular,kaltenbacher2022limiting}.
\\
\indent In order to be able to conduct such a limiting analysis one has to show that well-posedness can be obtained uniformly in $\tau$. Establishing uniform-in-$\tau$ energy bounds for \eqref{eqn:first_eq} is particularly challenging as the loss of the strong damping ($- \delta \Delta \psit$) makes it trickier to control the regularity of higher order terms when testing with, for example, $\psit$. 
In Sections~\ref{Sec:a_g_b},~\ref{Sec:a_eq_b}, and~\ref{Sec:r_one}, we discuss how specific structures of the equation can be leveraged to obtain $\tau$-uniform well-posedness. In particular, the following will play an important role: the comparison of $\aaa$ and $\bbb$, the regularity of the resolvent $\tfrakKone$, and the relationship between $\frakKone$ and $\frakKtwo$, which we assume to be of the form $$\frakKtwo = \frakR\Lconv \frakKone,$$ with $\frakR$ being an integrable function.
\\
\indent Thereafter, we show that as the relaxation time $\tau\searrow 0$, the solution of the model \eqref{eqn:first_eq} converges to the solution of a second-order-in-time equation with a dissipation of fractional type, thus connecting different models of linear fractional acoustics. The limiting procedure provides, incidentally, well-posedness for the limiting class of equations:
\begin{equation}
	\aaa\psi_{tt} - c^2 \bbb \Delta \psi -   \delta \rt^b \frakKtwo \Lconv \Delta\psi_{tt}   = f.
\end{equation}
\indent The aim, throughout, is to perform the uniform-in-$\tau$ analysis in as low a regularity setting as possible, thus we require minimal smoothness on the initial data and source term. We show that in such a setting, strong convergence to the limit is achieved and, by resorting to a creative tailored testing, a rate of convergence can be established in nonstandard norms. We thus complement the result of \cite{bongarti2020singular}, where equation \eqref{eq:MGT} is considered. This is expanded on in Section~\ref{Sec:vanishing_time_limit}.
\\
\indent 
We note that the estimates derived in Theorem~\ref{Prop:Wellposedness_a_g_b} and Proposition~\ref{Prop:wellposedness_GFE_laws} are also uniform with respect to the damping parameter $\delta$. These results could be used to study the $\delta$-limiting behavior of these equations, as done in~\cite{kaltenbacher2021inviscid} for integer-order models. We, nevertheless, do not pursue the inviscid limit analysis ($\delta \searrow 0$) in this work.
\\
\indent The rest of the paper is structured as follows. We give a physical motivation for the considered family of equations in Section~\ref{Sec:Modeling}. Section~\ref{Sec:preliminaries}, is concerned with adapting well-known results from functional analysis to the generalized fractional derivative setting. In Section~\ref{Sec:uniform_wellp}, we establish uniform well-posedness of \eqref{eqn:first_eq} under different assumptions on the constants $\aaa$ and $\bbb$, the kernel $\frakR$, and the regularity of the resolvent $\tfrakKone$. The limiting behavior of the model as the relaxation time vanishes is investigated in Section~\ref{Sec:vanishing_time_limit}. The main results of the said section are contained in Theorems~\ref{Prop:Cauchy_seq1} and \ref{Prop:Cauchy_seq2}.

\section{Acoustic modeling using generalized Maxwell--Cattaneo flux laws} \label{Sec:Modeling}
 We consider the flux law which incorporates the thermal relaxation as follows:
\begin{equation} \label{MC_flux_law_nonlocal}
	\bfq+ \tau^{a} \frakKone \Lconv \bfq_t= - \rt^b \kappa \frakKtwo\Lconv \nabla\theta_t.
\end{equation}
with  the kernels $\frakKone$ and $\frakKtwo$ being independent of $\tau$. Here $\bfq$ represents the heat flux and $\theta$, the temperature. The powers $a$ and $b$ attached to the relaxation times $\tau$ and $\rt$ are there to ensure the dimensional homogeneity of the law; see \cite{kaltenbacher2022limiting} for a more detailed discussion.
The notation $\genk\Lconv g$ stands for the Laplace convolution of $\genk$ and $g$ and is made precise in Section~\ref{Sec:preliminaries}.
This relation generalizes the well-known Maxwell--Cattaneo law~\cite{cattaneo1958forme}:
\begin{comment}
\end{comment}
\begin{equation} \label{MC_law}
	\tag{MC}
\begin{aligned}
	\bfq+\tau \bfq_t= - \kappa \nabla \theta ,
\end{aligned}
\end{equation}
as well as the Compte--Metzler fractional laws~\cite{compte1997generalized}, named hereafter GFE I, GFE II, GFE III, and GFE. The latter are generalized flux equations obtained for particular choices of kernels which we give in Table~\ref{table:r_kernels}.
The numbering of the equations follows that of~\cite{compte1997generalized,kaltenbacher2022time}. To each flux law, we associate a wave model name which will be used throughout to refer to the resulting wave equation. Table~\ref{table:r_kernels} contains also the expression of the resolvent $\tfrakKone$ as well as that of the kernel $\frakR$ verifying \[\frakKtwo = \frakR\Lconv \frakKone.\]
The penultimate column prescribes the range of the leading term fractional differentiation order $\alpha$ for which the heat flux law is defined. This range is taken from \cite{compte1997generalized} and, for GFE I, restricted using the results of \cite{zhang2014time}. This restriction for GFE I ensures that $\frakR \in L^1(0,T)$.
The last column refers to the theorems covering the singular $\tau$-limit analysis for the corresponding fractional MGT equation with $\aaa=\bbb =1$. The restriction on $\alpha$ for the fMGT wave model means that it is covered when $\alpha$ is larger than $\frac12$; see Section~\ref{Sec:a_eq_b} for more details on this requirement.
\begin{table}[H]
	\begin{adjustbox}{max width=\textwidth}
		\begin{tabular}{|c|c||c|c|c|c|c|c|c|l|}
			\hline&&&&&&&&&
			\\[-3mm]
			Flux law & Wave model &	$\frakR$ & $\frakKone$ & $\frakKtwo$ &$\tfrakKone$ &$a$&$b$ & Range of $\alpha$ & $\tau$-limit results
			\\[1mm]
			\hline\hline
			\ref{MC_law} \rule{0pt}{3ex} & MGT & $g_1$&$\delta_0$	&$g_{1}$ &	$g_{1}$ &$1$& 0& --& Theorems~\ref{Prop:Cauchy_seq1} and \ref{Prop:Cauchy_seq2} \\
			GFE I\rule{0pt}{3ex} & fMGT I & $g_{2\alpha-1}$	& $g_{1-\alpha}$ 	&$g_{\alpha}$	&	$g_{\alpha}$ &$\alpha$&$1-\alpha$& $(1/2,1)$& Theorem~\ref{Prop:Cauchy_seq2}\\
			GFE II\rule{0pt}{3ex} & fMGT II & $g_1$	& $g_{1-\alpha}$		&$g_{2-\alpha}$	&	$g_{\alpha}$ &$\alpha$&$\alpha-1$ &$(0,1)$& Theorems~\ref{Prop:Cauchy_seq1} and \ref{Prop:Cauchy_seq2} \\
			GFE III\rule{0pt}{3ex} & fMGT III & $g_\alpha$	&$\delta_0$			&$g_{\alpha}$&	$g_{1}$ &$1$ &$1-\alpha$ &$(0,1)$& Theorem~\ref{Prop:Cauchy_seq2}\\
			GFE \rule{0pt}{3ex} & fMGT & $g_\alpha$&$g_{1-\alpha}$	&$g_{1}$		&	$g_{\alpha}$ &$\alpha$& 0 &$(0,1)$& Theorem~\ref{Prop:Cauchy_seq2} if $\alpha > 1/2$\\
			\hline
		\end{tabular}
	\end{adjustbox}
	~\\[1mm]
	\caption{\small MGT and fMGT kernels with $g_\alpha$ defined in \eqref{eqn:def_g} and the theorems covering the singular limit analysis} \label{table:r_kernels}
\end{table}

$g_\alpha$ in Table~\ref{table:r_kernels} stands for the factional Caputo--Dzhrbashyan derivative kernel given by:
\begin{equation}\label{eqn:def_g}
	g_\alpha(t):= \dfrac{1}{\Gamma(\alpha)} t^{\alpha-1} \quad \textrm{for } \alpha>0.
\end{equation}
Note that $g_1 =1$ and that $g_{\alpha}\in L^1(0,T)$ for all $\alpha>0$.
\\
\indent We next discuss the derivation of an acoustic equation assuming the general flux law \eqref{MC_flux_law_nonlocal}. This follows by closely emulating the steps of \cite[Section 2]{kaltenbacher2022time}. 
In this derivation, we assume as is customary that initial values of quantities of interest are 0. 
In particular, this allows us to write 
\[\frakKone \Lconv \bfq_t = \Big(\frakKone \Lconv \bfq \Big)_t;\]
see~\cite[Corollary 3.7.3]{gripenberg1990volterra}. This assumption is however only made in this section where the justification of the model is sought. We argue here that this assumption, which is also made in, e.g., \cite{jordan2014second}, can lead to slightly different mathematical models of acoustics compared to~\cite{kaltenbacher2022time}.
Indeed, retracing the steps of~\cite[Section 2]{kaltenbacher2022time}, one may equivalently arrive at the wave equation: 
\begin{equation} 
	\begin{aligned}
		\begin{multlined}[t]
			\tau^a \big(\frakKone \Lconv \psi_{tt}\big)_t+\psi_{tt}
			 -  \tau^a  c^2 (\frakKone\Lconv\Delta \psi)_t -c^2 \Delta \psi - \rt^b \delta \frakKtwo \Lconv \Delta \psi_{tt}=0.
		\end{multlined}
	\end{aligned}
\end{equation}
We can then use the assumption of this section to argue that $\big(\frakKone\Lconv\Delta \psi \big)_t = \frakKone\Lconv\Delta \psi_t$. The reason we only do it for this term has to do with what is achievable in the analysis. In particular, when testing with $\psi_t$ we would like to argue that the term $- \frakKone\Lconv\Delta \psi_t$ dissipates energy even in the worst case scenario; see~\eqref{eqn:k1_positivity_lin} below.
Going forward, we study the abstract equation:
\begin{equation} 
	\begin{aligned}
		\begin{multlined}[t]
			\tau^a \big(\frakKone \Lconv \psi_{tt}\big)_t+ \aaa\psi_{tt}
			-  \tau^a  c^2  \frakKone\Lconv\Delta \psi_t  -c^2 \bbb\Delta \psi - \rt^b \delta \frakKtwo \Lconv \Delta \psi_{tt}=f,
		\end{multlined}
	\end{aligned}
\end{equation}
where $f$ is a source term and $\aaa,\, \bbb$ are positive real constants. The constants $\aaa$ and $\bbb$ were added so as to study how the order relation between them influences the behavior of the equation; see, e.g., Theorem~\ref{Prop:Wellposedness_a_g_b}. This equation is a generalization of the fractional MGT equations derived in \cite{kaltenbacher2022time}, with the reinterpretation of the leading-order derivative to be of generalized Riemann--Liouville type, which gets us closer to a desired rewriting; see discussion on page~\pageref{discussion:psit0}.
Obviously, if $\psitt(0) = 0$ is prescribed then the leading term can be expressed as 
\[\big(\frakKone \Lconv \psi_{tt}\big)_t = \frakKone \Lconv \psi_{ttt},\]
and the fractional MGT equations studied here coincide with those derived in~\cite{kaltenbacher2022time} as long as $\aaa = \bbb = 1$. Keeping the leading term of the form $\big(\frakKone \Lconv \psi_{tt}\big)_t$ gives us the freedom to cover a wider range of initial data (i.e., nonzero initial data), thus the choice was made to study the resulting acoustic equation with a Riemann--Liouville-type leading term.\label{discussion:RL_leading_term}
\section{Notation and preliminary theoretical results} \label{Sec:preliminaries}
Below, we will use the notation $A\lesssim B$ for $A\leq C\, B$ with a constant $C>0$ that may depend on the spatial domain $\Omega$, which we assume bounded and Lipschitz-regular in $\R^d$, but not time. If the constant depends on the time horizon $T$ we shall use the notation $A\lesssim_T B$. Above $d\geq1$ is the dimension of the space.

Let $X$ and $Y$ be two Banach spaces. We write
$X \hookrightarrow Y$ (respectively $X\dhookrightarrow Y$) for the continuous (respectively compact) embedding of $X$ into $Y$.\\
\indent In this work, $\mm(0,t)$ stands for the space of finite measures on $(0,t)$ and $\|\cdot\|_{\mm(0,t)}$ is the total variation norm associated to it; see \cite[Chapter 3]{gripenberg1990volterra} for more details.
\\\indent Recall that $\Lconv$ denotes the Laplace convolution, which should be interpreted as
\begin{align}
	&(\mathfrak{K}\Lconv g)(t)=\int_0^t \mathfrak{K} (s)\,g(t-s)\ds \quad &&\textrm{if }\genk, g \in L^1(0,t),\\
	&(\mathfrak{K}\Lconv g)(t)=\int_0^t \mathfrak{K}(\textup{d}s)\,g(t-s) \quad &&\textrm{if }\genk \in \mm(0,t),\ g \in L^1(0,t).
\end{align} 
\indent We shall also frequently use the symbol $\tilde{\mathfrak{K}}$ to denote the resolvent of a kernel/measure $\mathfrak{K}$, i.e., $\tilde{\mathfrak{K}} \Lconv \mathfrak{K} =1$. For discussions on the existence of such a resolvent, we refer to, e.g., \cite{gripenberg1980volterra,gripenberg1990volterra}.

Below, we give two lemmas allowing us to extract appropriately converging subsequences from bounded sequences in the space
\begin{equation} \label{def_Xfp}
X_\genk^p(0,T) = \{u \in L^p(0,T) \ |\ \genk \Lconv u_t \in L^p(0,T)\}, \qquad \textrm{where } \, 1 \leq p \leq \infty,
\end{equation}
endowed	with the norm
$$\|\cdot\|_{X_\genk^p(0,T)} = \big(\|u\|_{L^p}^p + \|(\genk\Lconv u_t)\|_{L^p}^p \big)^{1/p},$$ with the usual modification for $p=\infty$.
These lemmas, due to the construction of the spaces of interest, are similar to ones relating to Sobolev spaces (see e.g., \cite[Section 8.2]{brezis2010functional}).
\begin{lemma}[Compactness of $X_\genk^p(0,T)$] \label{Lemma:Caputo_seq_compact}
	Let $1\leq p\leq \infty$, $T>0$ and let $\genk \in \mm(0,T)$ be such that it has a resolvent $\tilde \genk \in L^{p'}(0,T)$ with $p' =  \dfrac{p}{p-1}$. 
	
	Then $X_\genk^p(0,T)$ is reflexive for $1<p<\infty$ and separable for $1 \leq p<\infty$.
	Furthermore, the unit ball of $X_\genk^p(0,T)$, $B_\genk^p$, is weakly sequentially compact for $1<p<\infty$. $B_\genk^\infty$ is weak-$*$ sequentially compact. 
	Additionally, $X_\genk^p(0,T)\hookrightarrow C[0,T]$.
\end{lemma}
Note that Lemma~\ref{Lemma:Caputo_seq_compact} can be extracted in the case $p=2$ with $\genk=g_\alpha$ (in particular for $\alpha>1/2$) from available norm equivalence and completion results~\cite[Theorems 2.2 and 2.5]{kubica2020time}. Here we are interested in more general kernels and in, among others, the case $p=\infty$.

To keep the notations somewhat compact, we will hereafter denote $X_\genk^p(0,T)$ simply by $X_\genk^p$, omitting the time range, so long as there is no confusion.
\begin{remark}[On the embedding $X_\genk^p\hookrightarrow {C[0,T]}$ ]
	It is not surprising that completeness (and in particular closedness) of the space $X_\genk^p$ depends on the regularity of $\tilde \genk$ and \emph{in fine} on the embedding $X^p_\genk \hookrightarrow C[0,T]$. In fact, for fractional derivative kernels, the discussion of \cite[Theorem 2.5]{kubica2020time} suggests that, for $\alpha\leq 1/2$, and $\big(u_{n}\big)_{n\geq 1}$ a sequence in $C^1[0,T] \subset X_\genk^p$ with $u_n(0)=0$ for all $n\geq 1$, it holds that \[g_{1-\alpha}* u_{nt} \to (g_{1-\alpha}* u)_t\]
	in $H^\alpha(0,T)$; see also norm equivalence result~\cite[Theorem 2.2]{kubica2020time}.
	On the other hand,
	\((g_{1-\alpha}* u)_t =  g_{1-\alpha}* u_t\) 
	if and only if $u(0) = 0$.
\end{remark}

\begin{proof} 
	
	We will first show that $X_\genk^p\hookrightarrow C[0,T]$. It suffices to see that since $\genk \Lconv  u_t \in L^p(0,T)$ and $\tilde \genk \in L^{p'}(0,T)$, then $\tilde \genk \Lconv \genk \Lconv u_t$ is continuous on $[0,T]$ and:
	\[u(t) = \tilde \genk \Lconv \genk \Lconv u_t + u(0).\]
	
	To show that $X_\genk^p$ is complete, take a Cauchy sequence $(u_n)_{n\geq1}\subset X_\genk^p$. Then $(u_n)_{n\geq1}$ and $(\genk\Lconv u_{nt})_{n\geq1}$ are Cauchy sequences in $L^p(0,T)$ and therefore converge to some limits $u$ and $g$, respectively, in $L^p(0,T)$. 
	
	We want to show that the limits are such that $\genk\Lconv u_t = g$ in a weak sense. To this end, let $n\geq 1$ and take an arbitrary $\phi \in C_c^1([0,T])$. We have, using \cite[Theorem 3.6.1(ix,xi)]{gripenberg1990volterra}, that
	\[
	\int_0^T (\genk\Lconv u_{nt}) \,\phi \ds = (\genk\Lconv u_{nt}\Lconv \tilde\phi)(T) = ( u_{nt}\Lconv (\genk\Lconv \tilde\phi))(T) =  \int_0^T u_{nt}(s)\, (\genk\Lconv \tilde \phi)(T-s) \ds,
	\]
	where $\tilde \phi(s) = \phi(T-s)$ for $s\in(0,T)$. Thus
	\begin{align}
	\int_0^T \genk\Lconv u_{nt} \,\phi \ds &= \int_0^T u_{n}(s)\, (\genk\Lconv \tilde\phi)_t(T-s) \ds - u_n(0)  (\genk\Lconv \tilde\phi)(T).
	\end{align}
	Passing to the limit with the aid of the embedding $X_\genk^p\hookrightarrow C[0,T]$, we obtain 
	\[
	\int_0^T g \,\phi \ds = \int_0^T u(s)\, (\genk\Lconv \tilde\phi)_t(T-s) \ds - u(0) (\genk\Lconv \tilde\phi)(T).
	\]
	We then reverse the operations on the right-hand side to find
	\[
	\int_0^T \genk\Lconv u_t \,\phi \ds = \int_0^T g \,\phi \ds.
	\]
	Thus, $u \in X^p_\genk$, $\genk\Lconv u_t = g$ weakly, and $\|u_n-u\|_{X_\genk^p} \to 0$ as $n\to\infty$.	
	
	For $1<p<\infty$, it can be shown that the space $X_\genk^p$ is reflexive, for example, by using the idea of an isometric operator; e.g., \cite[Proposition 8.1]{brezis2010functional}. 
	The compactness of the unit ball is obtained through the Eberlein--\v{S}mulian theorem \cite[Theorem 11.8]{clason2020introduction}. The space $X_\genk^p$ is also separable for all $1\leq p<\infty$.
	
	For $p=\infty$, let us consider a sequence $(u_n)_{n\geq1}$ in $B_\genk^\infty$. Then $(u_n, (\genk\Lconv u_{n})_t)_{n\geq1}$ is a sequence in \(L^\infty(0,T)\times L^\infty(0,T) = (L^1(0,T)\times L^1(0,T))^* \). By the Banach--Alaoglu theorem, $$(u_n, (\genk\Lconv u_{n})_t)_{n\geq1} \stackrel{}{\relbar\joinrel\rightharpoonup} (u,g) \quad \textrm{weakly-$*$ in} \quad (L^1(0,T)\times L^1(0,T))^*.$$
	Similarly to the proof of completeness, we can show that $\genk\Lconv u_t = g$ weakly, which concludes the proof.
	\qed
\end{proof}
Later on, we will additionally need compact embeddings using the space $X^p_\genk$. The following proposition is inspired by \cite[Theorem 3.1.1]{zheng2004nonlinear} as well as by \cite[Theorem 5]{simon1986compact}.

\begin{lemma}[Compact embedding using $X_\genk^p$]
	\label{Lemma:Caputo_compact_embedding}
	Let $X$, $Y$, and $Z$ be three Banach spaces such that $X$ and $Z$ are reflexive and $$X\dhookrightarrow Y \hookrightarrow Z.$$
Let $1<p\leq\infty$, and define
\[W_\genk = \{ u \, |\, u \in L^{p}(0,T; X), \ \genk\Lconv u_t \in L^{p}(0,T; Z) \},\] 
where $\genk$ verifies the assumptions of Lemma~\ref{Lemma:Caputo_seq_compact}. 
If $1<p<\infty$, then $W_\genk \dhookrightarrow L^{p}(0,T; Y)$. 

\noindent If $p=\infty$, and $\tilde \genk \in L^q(0,T)$ for some $q>1$, then $W_\genk \dhookrightarrow C([0,T]; Y)$. 
\end{lemma}
\begin{proof}
	Let $(u_n)_{n\geq1}$ be a (uniformly-in-$n$) bounded sequence in $W_\genk$. Without loss of generality, we may assume that $(u_n)_{n\geq1} \subset B_{W_\genk}$, where $B_{W_\genk}$ is the unit ball of $W_\genk$.
	
	Let $n\geq 1$. For the case $1 < p < \infty$, the proof follows by emulating \cite[Theorem 3.1.1]{zheng2004nonlinear} where the compact embedding of the Sobolev-like space \[\big\{  u \in L^{p}(0,T; X) \ \big| \ u_t \in L^{p}(0,T; Z) \big\}\] is studied. 
	The idea there is to argue that since $u_n \in L^p(0,T; Z)$ and $u_{nt} \in L^p(0,T; Z)$ then $u_n \in C([0,T]; Z)$. Note that we can argue a similar property on $u_n \in W_\genk$ owing to the embedding obtained in Lemma~\ref{Lemma:Caputo_seq_compact}:
	\[X_\genk^p(0,T; Z) \hookrightarrow C([0,T]; Z).\] 
	The rest of the proof follows then directly from the steps of \cite[Theorem 3.1.1]{zheng2004nonlinear} so we omit the details here.
	
	For the case $p = \infty$, our main ingredient is \cite[Theorem 5]{simon1986compact}, where a sufficient criterion  \cite[Eq. (8.3)]{simon1986compact} is given for the compact embedding to hold:
	\begin{equation}\label{time_criterion}
		\|u(t+h) - u(t)\|_{L^\infty(0,T-h;Z)} \to 0 \quad \textrm{as } h\searrow0, \quad \textrm{uniformly for $u$ in } B_{W_\genk}.
	\end{equation}
Note that the time-translation criterion~\eqref{time_criterion} can be interpreted as~\cite[Remark 3.1]{simon1986compact}:
$$\forall \varepsilon>0, \,\exists \eta \textrm{ such that: }\forall u \in B_{W_\genk},\, \forall h<\eta \textrm{ one has } \|u(t+h) - u(t)\|_{L^\infty(0,T-h;Z)}\leq \varepsilon.$$ 

Let $u \in B_{W_\genk}$. To show \eqref{time_criterion}, we use the expression 
\[u(t) = \tilde \genk \Lconv \genk \Lconv u_t(t) + u(0),\]
which is justified by Lemma~\ref{Lemma:Caputo_seq_compact}. For convenience, we will denote again $g = \genk \Lconv u_t$.
Thus, for an arbitrary $h>0$, we can write
	\begin{align}
		u(t+h) - u(t) &\,= \tilde \genk \Lconv g (t+h) - \tilde \genk \Lconv g (t)
		\\ &\,= \int_t^{t+h} \tilde \genk (t+h-s) g(s) \ds + \int_0^{t} (\tilde \genk (t+h-s) - \tilde \genk (t-s)) g(s) \ds.
	\end{align}
	It then follows that 
	\begin{align}
		&\|u(t+h) - u(t)\|_{L^\infty(0,T-h;Z)}\\  
		\leq & \Big(\esssup_{t\in (0,T-h)}\|\tilde \genk(\cdot)\|_{L^1(t,t+h)} + \esssup_{t\in (0,T-h)}\|\tilde \genk (t+h-\cdot) - \tilde \genk (t-\cdot)\|_{L^1(0,t)}\Big) \|g\|_{L^\infty(0,T-h;Z)}
		\\  
		\leq & \esssup_{t\in (0,T-h)}\|\tilde \genk(\cdot)\|_{L^1(t,t+h)} + \esssup_{t\in (0,T-h)}\|\tilde \genk (\cdot+h) - \tilde \genk (\cdot)\|_{L^1(0,t)},
	\end{align}	
	where in the last line we have used that $ \|g\|_{L^\infty(0,T-h;Z)} \leq 1$ since $u\in B_{W_\genk}$. In this last expression, $u$ and $g$ no longer appear and we can be sure that the convergence is uniform for $u$ in $B_{W_\genk}$.
	
	Because $\tilde \genk \in L^r(0,T) \subset L^1(0,T)$, and translation is continuous on $L^1(0,T)$, we infer that:
	\begin{equation}
	\esssup_{t\in (0,T-h)} \|\tilde \genk (\cdot+h) - \tilde \genk (\cdot)\|_{L^1(0,t)}  \to 0 \quad \textrm{as } h\searrow0.
	\end{equation}
	On the other hand, we have
	\[\esssup_{t\in (0,T-h)}\|\tilde \genk\|_{L^1(t,t+h)} = \esssup_{t\in (0,T-h)} h^{\frac{r-1}r} \|\tilde \genk \|_{L^r(t,t+h)} \to 0 \quad \textrm{as } h\searrow0.\]
	Thus, \eqref{time_criterion} holds which finishes the proof.
	\qed
\end{proof}
\begin{remark}[$X_\genk^p$ as a generalization of $W^{1,p}(0,T)$]
	Note that the requirements $\genk \in \mm(0,T)$
	as well as the existence of an $L^{p'}$-regular resolvent are flexible enough and allow us to see $X_\genk^p$ as a generalization of $W^{1,p}(0,T)$. 
	In particular if $\genk=\delta_0$, then $\tilde \genk = 1$ and $X_\genk^p = W^{1,p}(0,T)$.
	
	In general, for a kernel $\genk\in \mm(0,T)$ satisfying the assumptions of Lemma~\ref{Lemma:Caputo_seq_compact}, we have the inclusion:
	\[W^{1,p}(0,T) \subset X_\genk^p,\]
	owing to \cite[Theorem 3.6.1]{gripenberg1990volterra}.
\end{remark}

\section{\texorpdfstring{$\tau$}{tau}-uniform well-posedness analysis}\label{Sec:uniform_wellp}
We have now built the theoretical scaffolding to support our analysis. We next discuss the assumptions which relate to the properties of the kernels $\frakKone$ and $\frakKtwo$. Thereafter, we discuss the uniform-in-$\tau$ well-posedness of equation \eqref{eqn:first_eq}.
\subsection{Assumptions on the memory kernels} 
We formulate in this section the assumptions on the kernels needed for the upcoming analysis. The assumptions are verified for the wave models of interest which include the Moore--Gibson--Thompson equation as well as its fractionally relaxed counterparts contained in Table~\ref{table:r_kernels}. We first assume that
\leqnomode
\begin{equation}
	\label{eq:boundedness_assumption}
	\tag{$\bf\mathcal{A}_0$} 
	\begin{aligned}
		\frakKone \in \mm(0,T).
	\end{aligned} 
\end{equation}
This implies the boundedness of the operator
\begin{equation}
	\begin{aligned}
		\operatorname{T}_{\frakKone}\ : L^p(0,T) &\rightarrow L^p(0,T) 
		\\
	u&\mapsto \frakKone\Lconv u
	\end{aligned} 
\end{equation}
for all $1 \leq p \leq \infty$. The boundedness constant is given by the total variation norm $\|\frakKone\|_{\mm(0,T)}$; see \cite[Chapter 3]{gripenberg1990volterra} for more details.
\\
\indent When it comes to the resolvent of the leading measure, we assume the following:
\begin{equation}\label{eqn:resolvent}
	\textrm{there exists} \quad \tfrakKone \in L^{q}(0,T) \quad\textrm{for some} \quad q>1 \quad \textrm{such that} \quad \tfrakKone \Lconv \frakKone = 1  \tag{$\bf\mathcal{A}_1$}. 
\end{equation}
The assumption $q>1$ is needed so as to be able to use the compact embedding results of Lemma~\ref{Lemma:Caputo_compact_embedding} when showing that the solution of \eqref{eqn:first_eq} attains initial conditions; see proof of Theorem~\ref{Prop:Wellposedness_a_g_b}. The expression of the resolvent for the kernels/measures of interest is given in Table~\ref{table:r_kernels}.
\\
\indent We further assume that
\begin{equation}\label{eq:causality_eq}
	\textrm{there exists} \quad \mathfrak{r} \in L^1(0,T) \quad \textrm{such that} \qquad \frakKtwo = \mathfrak{r} \Lconv \frakKone. \tag{$\bf\mathcal{A}_2$} 
\end{equation}
This is the case for all (f)MGT kernels as can be seen from Table~\ref{table:r_kernels}.
In the general case, $\frakR$ can be found by viewing it as a solution to a Volterra integral equation of the first kind. One may relax \eqref{eq:causality_eq}, by allowing for $\frakR$ to be a finite measure on $(0,T)$ at the cost of increased technicality. We will not pursue this relaxation as it is not needed in the present setting.\\
\indent Note that as a result of \eqref{eq:boundedness_assumption} and \eqref{eq:causality_eq}, we obtain that $\frakKtwo \in L^1(0,T)$. Thus the operator $\operatorname{T}_{\frakKtwo}$ is bounded as well.

We assume the following ``positivity" properties form the kernels $\frakR$ and $\frakKone$:
\begin{equation}\label{eqn:causality_aaumption}
		\int_0^t (\mathfrak{r}*y)(s)y(s)\ds
		\geq 0, \quad y\in L^2(0,t)   \tag{$\bf\mathcal{A}_3$} ,
	\end{equation}
	\begin{equation}\label{eqn:k1_positivity_lin}
		\int_0^t (\frakKone \Lconv y)(s)y(s)\ds
		\geq 0, \quad y\in L^2(0,t)   \tag{$\bf\mathcal{A}_4$}.
\end{equation}
These assumptions are standard for fractional-type kernels and can be verified using a Fourier transform along the lines of \cite[Assumption \ensuremath{\bf {{A}}^{\textup{weak}}_1}]{kaltenbacher2022limiting}. 

We will further need some coercivity assumptions on $\frakKtwo$. \label{discussion_k2_assumption}
From Table~\ref{table:r_kernels}, we distinguish two cases on the behavior of $\frakKtwo$ as
$t\searrow 0$. Either the limit is $\infty$ or a finite nonnegative value. This inspires two alternative assumptions. 
\begin{itemize}
	\item  We assume that (at least) one of the following two assumptions holds:
	\begin{equation}\label{eqn:positivity}
		\int_0^t (\frakKtwo*y_t)(s)(y)(s)\ds
		\geq - C_{\frakKtwo} |y(0)|^2
		, \quad y\in C[0,t] \textrm{ and } \frakKtwo*y_t\in L^1(0,t) \tag{$\bf\mathcal{A}_5$}.
	\end{equation}
	Notice that $X_{\frakKone}^{p}(0,t)$ with $1\leq p\leq \infty$ provides the right space for this assumption. Indeed, if $y\in X_{\frakKone}^{p}(0,t)$, then $y\in C[0,t]$ by Lemma~\ref{Lemma:Caputo_seq_compact}, and $$\|\frakKtwo*y_t\|_{L^1(0,t)} \leq \|\frakR\|_{L^1(0,t)}\|\frakKone*y_t\|_{L^1(0,t)}.$$
	\item Alternatively, instead of using \eqref{eqn:positivity}, we may assume 
	\begin{equation}\label{eqn:positivity_nonsing}
		\tag{${\bf\mathcal{A}_5^{\textrm{alt}}}$}
		\begin{multlined}
			\frakKtwo \in W^{1,1}(0,T)\hookrightarrow C[0,T],\ \  \frakKtwo(0)\geq0, \ \textrm {and }\\
		\int_0^t (\mathfrak{K}_{2t}*y)(s)(y)(s)\ds
		\geq 0
		, \quad y\in L^2(0,t).
			\end{multlined}
	\end{equation}
	Here the idea is to use the following rewriting in the analysis:
	\begin{equation}\label{eqn:rewriting_k2}
	\frakKtwo*y_t+\frakKtwo\ y(0) =(\frakKtwo*y)_t = \mathfrak{K}_{2t}*y + \frakKtwo(0) y
	\end{equation}
	when $\frakKtwo$ is smooth, together with this assumption.
\end{itemize} 
For general kernels, assumption \eqref{eqn:positivity} can be verified along the lines \cite[Lemma B.1]{kaltenbacher2021determining} (using a density argument as it is stated in \cite{kaltenbacher2021determining} for $y\in W^{1,1}(0,t)$), while for \eqref{eqn:positivity_nonsing} one can use the Fourier transform employed in \cite[Assumption \ensuremath{\bf {{A}}^{\textup{weak}}_1}]{kaltenbacher2022limiting}.
In particular, \eqref{eqn:positivity} holds for heat kernels of wave models fMGT I and fMGT III, while \eqref{eqn:positivity_nonsing} holds for those of MGT, fMGT II, and fMGT. Combined, these two assumptions cover all the equations of interest in this paper. Note that if $y(0) = 0$, then \eqref{eqn:positivity_nonsing} implies \eqref{eqn:positivity}. 
\reqnomode
\subsection{(Uniform) well-posedness in the case $\aaa\geq\bbb$ }\label{Sec:a_g_b}
We aim here to provide a uniform-in-$\tau$ well-posedness result for
the general equation of interest given by \eqref{eqn:first_eq}
\begin{equation}
	\big(\tau^a\frakKone \Lconv \psi_{tt})_{t} + \aaa\psi_{tt} - c^2  \tau^a\frakKone \Lconv \Delta\psi_t - c^2 \bbb \Delta \psi -   \delta \rt^b \frakKtwo \Lconv \Delta\psi_{tt}   = f,
\end{equation}
supplemented with appropriate initial and boundary data. Let the coefficients in the equation be such that $\aaa \geq \bbb>0$.
The main idea of the upcoming proof is to use the two following equivalent rewritings of \eqref{eqn:first_eq}:
\begin{equation}\label{eq:nonunfiorm_a_geq_b}
	\big(\tau^a\frakKone \Lconv \psitt + \aaa\psit\big)_{t} - \frac\bbb\aaa c^2 \Delta \big(\tau^a\frakKone \Lconv \psi_t + \aaa\psi \big) - \frac{\aaa-\bbb}\aaa c^2  \tau^a\frakKone \Lconv \Delta\psi_t -   \delta \rt^b \frakKtwo \Lconv \Delta\psi_{tt} = f,
\end{equation}
and
\begin{equation}\label{eq:nonunfiorm_a_geq_b_2nd_option}
	\big(\tau^a\frakKone \Lconv \psitt + \bbb\psit\big)_{t} + (\aaa-\bbb)\psi_{tt} - c^2 \Delta \big(\tau^a\frakKone \Lconv \psi_t + \bbb\psi \big) -   \delta \rt^b \frakKtwo \Lconv \Delta\psi_{tt} = f.
\end{equation}
This way, the model can be seen as a damped wave equation for $\tau^a\frakKone \Lconv \psi_t + \aaa\psi$ and, alternatively, for $\tau^a\frakKone \Lconv \psi_t + \bbb\psi$, provided that the leading term satisfies
\begin{align}\label{eq:relation_psi1}
\big(\tau^a\frakKone \Lconv \psitt \big)_{t} = \big(\tau^a\frakKone \Lconv \psi_t\big)_{tt}.
\end{align}
With this view, $\big(\tau^a\frakKone \Lconv \psi_t + \aaa\psi\big)_t$ and  $\big(\tau^a\frakKone \Lconv \psi_t + \bbb\psi\big)_t$ are natural test functions. 

Equality \eqref{eq:relation_psi1} can be ensured by requiring that $\psit(0) = 0$ owing to \cite[Corollary 3.7.3]{gripenberg1990volterra}. The identity is stated there for an absolutely continuous function (i.e., for $\psi_t\in W^{1,1}(0,T)$), but retracing the proof of \cite[Theorem 3.7.1]{gripenberg1990volterra} shows that it is sufficient that $\psit$ be continuous and that $\tau^a\frakKone \Lconv \psitt \in L^1(0,T)$, both of which we ensure below through the spaces $X_{\taua\frakKone}^p(0,T)$ (with $p=\infty$, except in Proposition~\ref{Prop:wellposedness_GFE_laws} where $p=2$). 

The assumption that one of the initial conditions should be zero can often be found in the analysis of fractional PDEs; see e.g., \cite[Section 7]{kaltenbacher2022time} and \cite[Proposition 3.2]{kaltenbacher2021determining} for a similar requirement. Note that $\psit(0) = 0$ is not needed if $\frakKone = \delta_0$ (MGT, and fMGT III models), as then relation \eqref{eq:relation_psi1} always holds.\label{discussion:psit0}

For the coming results we define the solution space
\begin{equation}
	\calX^\infty = \{ \psi \in X_{\taua\frakKone}^\infty(0,T; \Honezero)\,|\, \psit \in X_{\taua\frakKone}^\infty(0,T; L^2(\Omega))\},
\end{equation}
where $X_{\taua\frakKone}^\infty$ is defined in \eqref{def_Xfp}. This space fulfills the requirements of Lemma~\ref{Lemma:Caputo_seq_compact} thanks to \eqref{eqn:resolvent}.
The next result establishes well-posedness of an initial-value-boundary problem of \eqref{eqn:first_eq}. The uniformity in $\tau$ of the result will depend on whether $\aaa$ is strictly larger than $\bbb$.

\begin{theorem}\label{Prop:Wellposedness_a_g_b}
	Let $T>0$, $\aaa\geq \bbb >0$, and $\tau>0$. 
	Let Assumptions \eqref{eq:boundedness_assumption}
	--\eqref{eqn:k1_positivity_lin}, and \eqref{eqn:positivity} or \eqref{eqn:positivity_nonsing} (see discussion above) hold. Then, given initial data \[(\psi_0,\psi_1,\psi_2^{\frakKone}) \in H_0^1(\Omega) \times \{0\}\times L^2(\Omega)\] and a source term $f\in L^1(0,T;L^2(\Omega))$, 
	there is a unique $\psi \in \calX^\infty$ which solves 
	\begin{equation}\label{eqn:ibvp_eqn}
		\begin{aligned}
			\begin{multlined}[t] -\taua\intT \big(\frakKone * \psitt , v_t\big)_{L^2} \ds - \intT(\aaa\psit, v_t)_{L^2}\ds + c^2  \intT(\bbb \nabla \psi, \nabla v)_{L^2}\ds \\
				+\tau^a c^2 \intT\big(\frakKone \Lconv \nabla \psit , \nabla v\big)_{L^2}\ds - \rt^b \delta  \intT(\frakKtwo \Lconv \nabla \psit, \nabla v_t)_{L^2} \ds
				\\= - \taua \big(\psitwo , v(0)\big)_{L^2} + \intT(f, v)_{L^2} \ds,
			\end{multlined}	
		\end{aligned}		
	\end{equation}
	for all $v\in H^1(0,T;H_0^1(\Omega))$ such that $v(T) = 0$, with
	\[ (\psi,\psit) \Big|_{t=0} = (\psi_0,0).\]
	Furthermore, for almost all $t\in(0,T)$, the solution satisfies
	\begin{equation}
		\begin{multlined}
			\|\tau^a \frakKone\Lconv\psi_{tt}(t) \|^2_{L^{2}(\Omega)}+
			 \|\psi_{t}(t)\|^2_{L^2(\Omega)}+  \|\tau^a \frakKone\Lconv\nabla\psi_t(t)\|^2_{L^2(\Omega)}+\| \nabla \psi(t)\|^2_{L^2(\Omega)}\\ \lesssim \tau^{2a} \|\psitwo\|^2_{L^2(\Omega)}+\|\nabla\psi_0\|^2_{L^2(\Omega)}+\|f\|^2_{L^1(0,t;L^2(\Omega))},
		\end{multlined}
	\end{equation} 
where the hidden constant is independent of $\delta$ and $T$.
If $\aaa > \bbb$, then the hidden constant is also independent of $\tau$.
\end{theorem} 
\begin{proof}
	We use a standard Galerkin procedure to construct an approximate solution; see, e.g., \cite[Chapter 7]{evans2010partial} and \cite{kaltenbacher2022time}. Given an orthogonal basis $\{\phi_n\}_{n \geq 1}$ of $V=H_0^1(\Om)$, let $V_n=\text{span}\{\phi_1, \ldots, \phi_n\} \subset V$ and
	\begin{equation}
		\begin{aligned}
			\psi^{(n)}(t) = \sum_{i=1}^n \xin (t) \phi_i.
		\end{aligned}	
	\end{equation}
	Choose the approximate initial data
	\begin{equation}
		\psi^{(n)}_0=  \sum_{i=1}^n \xi_i^{(0, n)} \phi_i,\quad \psi^{(n)}_1= 0, \quad  \psi^{{\frakKone}(n)}_2 = \sum_{i=1}^n \xi_i^{(2, n)} \phi_i \in V_n,
	\end{equation}
	such that
	\begin{equation} \label{convergence_approx_initial_data_z_form}
		\begin{aligned}
			\psi^{(n)}_0 \rightarrow \psi_0 \  \text{in} \ H_0^1(\Om), \ \text{and } \, \psi^{(n)}_2 \rightarrow \psi_2 \  \text{in} \ L^2(\Om), \textrm{ as } \ n \rightarrow \infty.
		\end{aligned}
	\end{equation}
	For each $n \in \N$, the system of Galerkin equations is given by
	\begin{equation}
		\begin{aligned}
			\begin{multlined}[t]\taua	\sum_{i=1}^n (\frakKone * \xitt)_t(t) (\phi_i, \phi_j)_{L^2}+ \sum_{i=1}^n \xitt (\aaa\phi_i, \phi_j)_{L^2}+c^2  \sum_{i=1}^n \xin (\bbb \Delta \phi_i, \phi_j)_{L^2} \\
				+\tau^a c^2 \sum_{i=1}^n (\frakKone * \xit)(t) (\nabla \phi_i, \nabla \phi_j)_{L^2} +\rt^b \delta \sum_{i=1}^n (\frakKtwo * \xitt)(t) (\nabla \phi_i, \nabla \phi_j)_{L^2} \\
				= (f(t), \phi_j)_{L^2}
			\end{multlined}	
		\end{aligned}		
	\end{equation}
	for a.e.\ $t \in (0,T)$ and all $j \in \{1, \ldots, n\}$. With $\bxi = [\xin_1 \ \ldots \  \xin_n]^T$, we can write this system in matrix form 
	\begin{equation}
		\left \{	\begin{aligned}
			& \taua M	(\frakKone * \bxitt)_t + M_\aaa \bxitt+ K_\bbb \bxi+\tau^a c^2 K  \frakKone* \bxit + \rt^b \delta K \frakKtwo *\bxitt = \boldsymbol{f}, \\[1mm]
			& (\bxin, \bxit, \frakKone\Lconv\bxitt)\vert_{t=0} = (\boldsymbol{\xi_0}, \boldsymbol{0}, \boldsymbol{\xi_2}^{\frakKone}),
		\end{aligned} \right.
	\end{equation}
	where $(\boldsymbol{\xi_0},  \boldsymbol{\xi_2}^{\frakKone})=([\xi_1^{(0,n)} \, \ldots \, \xi_n^{(0, n)}]^T,\, [\xi_1^{(2,n)} \, \ldots \, \xi_n^{(2, n)}]^T)$.

	To prove that the Galerkin system is uniquely solvable, we introduce the new unknown \[\hat{\boldsymbol\chi} = (\frakKone\Lconv\boldsymbol{\xi}_{tt})_t.\]
	We can then rewrite the semi-discrete fractional derivative system using
	\begin{equation}
		\begin{aligned}
			\boldsymbol\xi_{tt} =& \, {\tfrakKone}\Lconv \hat{\boldsymbol\chi} + \boldsymbol\xi_{2}^{\frakKone}{\tfrakKone}\\
			\boldsymbol\xi_t =& \,1 \Lconv {\tfrakKone}\Lconv \hat{\boldsymbol\chi} + 1 \Lconv {\tfrakKone}\boldsymbol\xi_{2}^{\frakKone} 
			\\
			\boldsymbol\xi =& \,1 \Lconv 1 \Lconv {\tfrakKone}\Lconv \hat{\boldsymbol\chi} + 1 \Lconv 1 \Lconv {\tfrakKone} \boldsymbol\xi_{2}^{\frakKone}
			+ \boldsymbol\xi_0.
		\end{aligned}
	\end{equation}
This yields
\begin{equation}
	\begin{aligned}
		\begin{multlined}[t] \taua	\hat{\boldsymbol\chi}+ M^{-1} M_\aaa {\tfrakKone}* \hat{\boldsymbol\chi} + M^{-1}K_\bbb \, 1*1*{\tfrakKone}*\hat{\boldsymbol\chi}
			+\tau^a c^2 M^{-1}K \, 1* 1*\hat{\boldsymbol\chi}\\+ \rt^b \delta M^{-1}K \, \frakKtwo *{\tfrakKone}* \hat{\boldsymbol\chi}  = {\boldsymbol{\tilde f}}.
		\end{multlined}
	\end{aligned}
\end{equation}
Here, $M$ and $K$ are the mass and stiffness matrices whose entries are given by
\begin{align}
		M_{i,j} = \intO \phi_i\,\phi_j \dx, \qquad K_{i,j} = \intO \nabla\phi_i \cdot \nabla\phi_j \dx.
\end{align}
$M_\aaa$ and $K_\bbb$ are weighted matrices which are simply expressed as 
$M_\aaa = \aaa M$ and $K_\bbb = \bbb K$ (because $\aaa$ and $\bbb $ are constants). Above, the source term has the following form:\\ 
\begin{equation}
	\begin{aligned}
		{\boldsymbol{\tilde f}}=\, 	\begin{multlined}[t]  M^{-1}\boldsymbol{f}-M^{-1} M_\aaa \boldsymbol\xi_{2}^{\frakKone}{\tfrakKone}-M^{-1}K_\bbb(1 \Lconv 1 \Lconv {\tfrakKone} \boldsymbol\xi_{2}^{\frakKone} + \boldsymbol\xi_0)\\-\tau^a c^2 M^{-1}K \,(1 \Lconv 1) \, \boldsymbol\xi_{2}^{\frakKone}  -\rt^b \delta M^{-1}K (\frakKtwo \Lconv{\tfrakKone})\boldsymbol\xi_{2}^{\frakKone}.	\end{multlined} 
	\end{aligned}
\end{equation}
Notice that $\boldsymbol{\tilde f} \in L^1(0,T)$.
	By \cite[Theorem 2.3.5]{gripenberg1990volterra}, the system has a unique solution $\hat{\boldsymbol\chi} \in L^1(0,T)$.
	In turn, we obtain that $\boldsymbol{\xi} \in \{\boldsymbol{u} \in W^{2,1}(0,T) \ |\ \frakKone\Lconv\boldsymbol{u}_{tt} \in W^{1,1}(0,T)\}$.
	Thus, $\psi^{(n)} \in \{u \in W^{2,1}(0,T;V_n) \ |\ \frakKone\Lconv u_{tt} \in W^{1,1}(0,T;V_n)\}$.
\subsubsection*{Energy estimates}
In what follows we drop the superscript $n$ and simply refer to the semi-discrete problem's solution as $\psi$.
To simplify the presentation, let us also introduce the variable 
\[z := \tau^a\frakKone \Lconv \psi_t + \aaa\psi,\]
for which we view \eqref{eq:nonunfiorm_a_geq_b} as a damped wave equation. Introducing the $z$ variable (or some variation thereof) is fairly standard in the analysis of the MGT equation and can be traced back to some of the earliest  papers on the topic; see, e.g.,~\cite{kaltenbacher2011wellposedness}.
 
We test \eqref{eq:nonunfiorm_a_geq_b} with $z_t$ to obtain \color{black}
\begin{equation}\label{eq:energy_est_one}
	\begin{multlined}
		\|z_{t}(s)\|^2_{L^2(\Omega)}\Big|^t_0 + \frac\bbb\aaa c^2 \|\nabla z(s) \|^2_{L^2(\Omega)}\Big|^t_0 + \frac{\aaa-\bbb}\aaa c^2  \intTO  \tau^a\frakKone \Lconv \nabla\psi_t \cdot \nabla z_t \dxs  \\+    \delta \rt^b \intTO  \frakKtwo \Lconv \nabla\psi_{tt} \cdot \nabla z_t \dxs  = \intTO  f\, z_{t} \dxs.
	\end{multlined}
\end{equation}

We bound the third term of the left-hand side from below as follows:
\begin{equation}
	\begin{multlined}
		 \frac{\aaa-\bbb}\aaa c^2  \intTO  \tau^a\frakKone \Lconv \nabla\psi_t \cdot \nabla z_t \dxs\\ = \frac{\aaa-\bbb}\aaa c^2 \|\tau^a\frakKone \Lconv \nabla\psi_t(s)\|^2_{L^2(\Omega)}\Big|^t_0 + (\aaa-\bbb) \,c^2  \intTO  \tau^a\frakKone \Lconv \nabla\psi_t \cdot \nabla\psi_t \dxs \geq 0,
	\end{multlined}
\end{equation}
where the last inequality is obtained using $\aaa-\bbb\geq0$ and \eqref{eqn:k1_positivity_lin}.
For the fourth term in~\eqref{eq:energy_est_one}, we have 
\begin{equation}
	\begin{multlined}
		\delta \rt^b \intTO  \frakKtwo \Lconv \nabla\psi_{tt} \cdot \nabla z_t\\= \delta \rt^b  \underbrace{\intTO\frakR \Lconv \frakKone \Lconv \nabla\psi_{tt} \cdot \big(\tau^a\frakKone \Lconv \nabla\psi_{tt} \big) \dxs }_{\geq 0, \textrm{ due to }\eqref{eqn:causality_aaumption}}+\delta\rt^b\intTO\frakKtwo \Lconv \nabla\psi_{tt} \cdot \aaa\nabla\psi_t \dxs.
	\end{multlined}
\end{equation}
Note that we have enough regularity in the semi-discrete setting ($\frakKone\Lconv\psi_{tt} \in W^{1,1}(0,T;V_n)$) such that using \eqref{eqn:causality_aaumption} is allowed.
Since we are in the setting $\psit(0) = \psi_1 =0$, assumption \eqref{eqn:positivity_nonsing} implies \eqref{eqn:positivity}, such that with either assumption we obtain
\begin{equation}
	\begin{multlined}
	\delta \rt^b \intTO  \frakKtwo \Lconv \nabla\psi_{tt} \cdot \aaa\nabla\psit \dxs \geq 0.
	\end{multlined}
\end{equation}
Further, H\"older's inequality yields
\begin{align}
\intTO  f \, z_t \dxs
\leq\, & \|f\|_{L^1(0,t;L^2(\Omega))} \|z_{t}\|_{L^\infty(0,t;L^2(\Omega))} \\ \leq\,& \frac1{4\varepsilon}\|f\|_{L^1(0,t;L^2(\Omega))}^2 + \varepsilon \|z_{t}\|_{L^\infty(0,t;L^2(\Omega))}^2,
\end{align}
for all $\varepsilon>0$. Here, we have used Young's inequality for the last step.

Piecing the different estimates together with a sufficiently small $\varepsilon$, we obtain 
\begin{equation}
	\begin{multlined}
		\|z_t(t)\|^2_{L^2(\Omega)}+ \frac\bbb\aaa c^2 \|\nabla z (t)\|^2_{L^2(\Omega)}   \lesssim \|f\|^2_{L^1(0,t;L^2(\Omega))} +  \|\big(\tau^a\frakKone \Lconv \psi_{tt}\big)(0)\|^2_{L^2(\Omega)} \\+ \|\nabla\big(\tau^a\frakKone \Lconv \psi_{t}\big)(0)\|^2_{L^2(\Omega)} + \|\nabla\psi(0)\|^2_{L^2(\Omega)}.
	\end{multlined}
\end{equation}

Using \cite[Theorem 3.6.1]{gripenberg1990volterra}, we have a Young's convolution inequality analogous for convolution with measures (where continuity is preserved due to \cite[Corollary 3.6.2]{gripenberg1990volterra} and $\psit(0)= \psi_1=0$), thus 
\[\|\nabla\big(\frakKone \Lconv  \psi_{t}\big)(0)\|_{L^2(\Omega)} \leq \|\frakKone\|_{\mm(0,0)} \|\nabla\psi_{t} (0)\|_{L^2(\Omega))} = 0. \]
Therefore, we obtain the energy estimate:
\begin{equation}
	\begin{multlined}\label{eqn:linear_z_estimate}
	\|z_t(t)\|^2_{L^2(\Omega)}+ \frac\bbb\aaa c^2 \|\nabla z (t)\|^2_{L^2(\Omega)} \lesssim \|f\|^2_{L^1(0,t;L^2(\Omega))} + \tau^{2a}  \|\psitwo\|^2_{L^2(\Omega)} + \|\nabla\psi_0\|^2_{L^2(\Omega)},
	\end{multlined}
\end{equation}
which is uniform in both $n$ and $\tau$. The hidden constant is also independent of $\delta$ and $T$.

To obtain estimates on $\frakKone\Lconv \psi_t$ and $\psi$, we use a bootstrap argument. We distinguish two cases below: either $\aaa = \bbb$ or $\aaa > \bbb$.

\textbullet \ {\it Bootstrap argument when $\aaa = \bbb$.}
Consider the auxiliary problem:
\[\tau^a\frakKone \Lconv \psi_t + \aaa\psi = z \quad\textrm{a.e. in } \Omega,\] with $z \in W^{1,\infty}(0,T;L^2) \cap L^\infty(0,T;H^1)$ and $\psi(0)= \psi_0$. 
Convolving with $\tfrakKone$ yields the following Volterra integral equation of the second kind:
\begin{equation}\label{eq:bootstrap_eq}
\tau^a \psi+ \aaa {\tfrakKone} \Lconv \psi = \tau^a \psi_0  + {\tfrakKone} \Lconv z.
\end{equation}
Thus, according to existence theory of Volterra equations of the second kind~\cite[Theorem 2.3.5]{gripenberg1990volterra}, equation \eqref{eq:bootstrap_eq} has a unique solution which satisfies (using the variation of constants formula given in the cited theorem):
\begin{equation}
	\|\psi\|_{W^{1,\infty}(0,T;L^2(\Omega))} \lesssim \|\psi_0\|_{L^2(\Om)} + \|z\|_{W^{1,\infty}(0,T;L^2(\Omega))},
\end{equation}
and 
\begin{equation}
	\|\psi\|_{L^\infty(0,T;H^1(\Omega))} \lesssim \|\psi_0\|_{H^1(\Om)} + \| z\|_{L^\infty(0,T;H^1(\Omega))},
\end{equation}
where the hidden constant does not depend on $T$. 
Together with \eqref{eqn:linear_z_estimate}, this yields
\begin{equation}
	\begin{multlined}
		\|\tau^a\frakKone \Lconv \psi_{tt} (t)\|^2_{L^2(\Omega)}+ \|\psi_{t}(t)\|^2_{L^2(\Omega)}+ \|\tau^a\frakKone \Lconv \nabla\psi_t(t)\|^2_{L^2(\Omega)} + \|\nabla\psi(t)\|^2_{L^2(\Omega)} \\ \lesssim \|f\|^2_{L^1(0,t;L^2(\Omega))} +  \tau^{2a}\|\psitwo\|^2_{L^2(\Omega)} + \|\nabla\psi_0\|^2_{L^2(\Omega)},
	\end{multlined}
\end{equation}
for almost every $t\in (0,T)$, where the hidden constant is independent of $\delta$ and $T$ but does depend on $\tau$.
\color{black}

\textbullet \ {\it Bootstrap argument when $\aaa > \bbb$.} When $\aaa \neq \bbb$, testing the equation (rewritten as \eqref{eq:nonunfiorm_a_geq_b_2nd_option}) with $\big(\tau^a\frakKone \Lconv \psi_t + \bbb\psi\big)_t$ yields, under the same assumptions as before:
\begin{equation}
	\begin{multlined}\label{eqn:linear_z_estimate_2}
		\|\big(\tau^a\frakKone \Lconv \psi_t + \bbb\psi\big)_{t}(t)\|^2_{L^2(\Omega)}+ c^2 \|(\tau^a\frakKone \Lconv \nabla\psi_t + \bbb\nabla\psi)(t)\|^2_{L^2(\Omega)}  \\\lesssim \|f\|^2_{L^1(0,t;L^2(\Omega))} +  \tau^{2a}\|\psitwo\|^2_{L^2(\Omega)} + \|\nabla\psi_0\|^2_{L^2(\Omega)},
	\end{multlined}
\end{equation}
where the hidden constant does not depend on $\tau$.
Combining \eqref{eqn:linear_z_estimate_2} and \eqref{eqn:linear_z_estimate}, we obtain
\begin{equation}
	\begin{multlined}
		\|\tau^a\frakKone \Lconv \psi_{tt} (t)\|^2_{L^2(\Omega)}+ \|\psi_{t}(t)\|^2_{L^2(\Omega)}+ \|\tau^a\frakKone \Lconv \nabla\psi_t(t)\|^2_{L^2(\Omega)} + \|\nabla\psi(t)\|^2_{L^2(\Omega)} \\ \lesssim \|f\|^2_{L^1(0,t;L^2(\Omega))} +  \tau^{2a}\|\psitwo\|^2_{L^2(\Omega)} + \|\nabla\psi_0\|^2_{L^2(\Omega)},
	\end{multlined}
\end{equation}
where, this time, the hidden constant does neither depend on $\tau$ nor on $\delta$ nor on final time $T$.

\subsubsection*{Passing to the limit}
 From the previous analysis and thanks to Lemma~\ref{Lemma:Caputo_seq_compact}, we conclude that there is a subsequence (not relabeled), such that 
 \begin{alignat}{2}
	\psi^{(n)} &\stackrel{}{\relbar\joinrel\rightharpoonup} \psi \quad &&\textrm{weakly-$*$ in } L^\infty(0,T; H_0^1(\Om)),\\
	\psi^{(n)}_t &\stackrel{}{\relbar\joinrel\rightharpoonup} \psi_t \quad &&\textrm{weakly-$*$ in } L^\infty(0,T; L^2(\Om)),\\
	\frakKone \Lconv \psi^{(n)}_t &\stackrel{}{\relbar\joinrel\rightharpoonup} \frakKone \Lconv \psi_t \quad &&\textrm{weakly-$*$ in } L^\infty(0,T; H_0^1(\Om)),\\
	\frakKone \Lconv \psi^{(n)}_{tt} &\stackrel{}{\relbar\joinrel\rightharpoonup} \frakKone \Lconv \psi_{tt} \quad &&\textrm{weakly-$*$ in } L^\infty(0,T; L^2(\Om)).
	\shortintertext{By $\frakKtwo = \frakR\Lconv \frakKone$ and Young's inequality:}
\frakKtwo \Lconv \psi^{(n)}_{tt} = \frakR \Lconv \frakKone \Lconv \psi^{(n)}_{tt} &\stackrel{}{\relbar\joinrel\rightharpoonup} \frakKtwo \Lconv \psi_{tt} \quad &&\textrm{weakly-$*$ in } L^\infty(0,T; L^2(\Om)),\\
\frakKtwo \Lconv \psi^{(n)}_{t} &\stackrel{}{\relbar\joinrel\rightharpoonup} \frakKtwo \Lconv \psi_{t} \quad &&\textrm{weakly-$*$ in } L^\infty(0,T; \Honezero). 
\end{alignat}
By \eqref{eqn:resolvent} and Lemma~\ref{Lemma:Caputo_compact_embedding}, there is a subsequence (again not relabeled), such that 
\begin{equation}\label{strong_conv_GFEII_z_form}
	\begin{alignedat}{4} 
		\psi^{(n)} &\longrightarrow \psi && \quad \text{ strongly}  &&\text{ in } &&C([0,T]; L^2(\Om),  \\
		\psit^{(n)} &\longrightarrow \psit && \quad \text{ strongly}  &&\text{ in } && C([0,T]; H^{-1}(\Om)).
	\end{alignedat} 
\end{equation} 
This is enough to pass in the weak form to the limit in the semi discrete equation and show that $\psi$ solves:
	\begin{equation}
	\begin{aligned}
		\begin{multlined}[t]- \taua\intT \big(\frakKone * \psitt , v_t\big)_{L^2} \ds-  \intT(\aaa\psit, v_t)_{L^2}\ds + c^2  \intT(\bbb \nabla \psi, \nabla v)_{L^2}\ds \\
			+\tau^a c^2 \intT\big(\frakKone \Lconv \nabla \psit , \nabla v\big)_{L^2}\ds -\rt^b \delta  \intT(\frakKtwo \Lconv \nabla \psit, \nabla v_t)_{L^2} \ds
			\\=  - \taua \big(\psitwo , v(0)\big)_{L^2}  + \intT(f, v)_{L^2} \ds,
		\end{multlined}	
	\end{aligned}		
\end{equation}
for all $v \in H^1(0,T;H_0^1(\Om))$ such that $v(T) = 0$. Here, we have used $\frakKtwo \Lconv \psitt = (\frakKtwo \Lconv \psit)_t$.
Note that due to Lemma~\ref{Lemma:Caputo_seq_compact}, we have  
\begin{equation}\label{continuity_GFEII_z_form}
	\psi \in C([0,T]; H_0^1(\Om)) \quad \psi_t \in C([0,T]; L^2(\Om)).
\end{equation}
From \eqref{convergence_approx_initial_data_z_form}, \eqref{strong_conv_GFEII_z_form}, \eqref{continuity_GFEII_z_form}, and uniqueness of limits we conclude that
\begin{equation}\label{init_data_GFEII_first_attainment}
	(\psi,\psi_t)|_{t=0} = (\psi_0, 0).
\end{equation}

\subsubsection*{Uniqueness}
Assume that initial data $\psi_0 = \psi_1 = \psitwo = 0$  and the source term $f=0$. We want to show that necessarily $\psi = 0$ to prove uniqueness. 
Obviously, if $\psi_t$ and $\frakKone\Lconv \psitt$ were smooth enough, we could test with them similarly to above and the proof would be direct. However, that is not the case. To solve this issue, we view equation~\eqref{eqn:first_eq} again as a wave equation for $z=\taua \frakKone*\psi_t+\bbb \psi$, and following the approach of \cite[Theorem 7.2.4]{evans2010partial}, 
we introduce valid test functions.
Fix $ 0\leq t' \leq T$ and set 
\begin{equation}
	w (t)=
	\left\{\begin{array}{ll}
		\int_t^{t'} \psi(s) \ds \quad  & \textrm{ if} \ \ 0\leq t\leq t'\\
		0 \quad  & \textrm{ if} \ \ t'\leq t\leq T.
	\end{array}\right.
\end{equation} 
We define the convolution-term analogous to $w$ as
\begin{equation}
	w_{\frakKone} (t)= 
	\left\{\begin{array}{ll}
		\int_t^{t'} \frakKone\Lconv\psi_t(s) \ds \quad  & \textrm{ if} \ \ 0\leq t\leq t'\\
		0 \quad  & \textrm{ if} \ \ t'\leq t\leq T.
	\end{array}\right.
\end{equation}  
Note that $w_{\frakKone} (t)= \frakKone \Lconv w (t)$ for $0\leq t\leq t'$. Both $w$, $w_{\frakKone} \in H^1(0,T;\Honezero)$ with $ w(t') = w_{\frakKone}(t') =0$ and are thus valid test functions. We can then write
\begin{equation}\label{eqn:uniqueness_GFEII_z_form_1}
	\begin{aligned}
		\begin{multlined}[t]-  \int_0^{t'} \big((\taua \frakKone * \psitt) , \taua w_{\frakKone t} + \bbb w_t\big)_{L^2} \ds - \int_0^{t'} (\aaa\psit, \taua w_{\frakKone t} + \bbb w_t)_{L^2}\ds\\+ \int_0^{t'} c^2  (\bbb \nabla \psi,  \taua \nabla w_{\frakKone} + \bbb \nabla w)_{L^2} \ds 
			+ c^2 \int_0^{t'} \big(\tau^a\frakKone \Lconv \nabla \psit ,  \taua \nabla w_{\frakKone} + \bbb \nabla w\big)_{L^2} \ds\\ -\rt^b \delta  \int_0^{t'} (\frakKtwo \Lconv\nabla  \psit,  \taua \nabla w_{\frakKone t} + \bbb \nabla w_t)_{L^2} \ds = 0,
		\end{multlined}	
	\end{aligned}		
\end{equation}
where we have used $\psi_0 = \psi_1 = \psitwo = 0$. % as well as $w(t')=w_{\frakKone}(t') = 0$.
Furthermore, we use that for $0 \leq t \leq t'$
\begin{align}\label{w_identity}
	w_t &= - \psi,\\
	w_{\frakKone t} &= - \frakKone\Lconv \psi_t,
\end{align}
such that using \eqref{eqn:positivity} (or \eqref{eqn:positivity_nonsing}) and \eqref{eqn:causality_aaumption}, we obtain 
\begin{align}
-\int_0^{t'} (\frakKtwo \Lconv\nabla  \psit, \nabla w_t)_{L^2} \ds & \geq 0, \\
-\int_0^{t'} (\frakKtwo \Lconv\nabla  \psit, \nabla w_{\frakKone t})_{L^2} \ds & \geq 0.
\end{align}
Thus \eqref{eqn:uniqueness_GFEII_z_form_1} becomes
\begin{equation}\label{eqn:uniqueness_GFEII_z_form_4}
	\begin{aligned}
		\begin{multlined}[t] \|\taua \frakKone * \psit + \bbb \psi\|^2_{L^2(\Om)}(t')+c^2 \|\taua \nabla w_{\frakKone} + \bbb \nabla w\|^2_{L^2(\Om)}(0) \\- c^2(\aaa-\bbb) \int_0^{t'} (\psi_t,\bbb w_t +\taua  w_{\frakKone t})_{L^2(\Om)}\ds \leq 0.
		\end{multlined}	
	\end{aligned}		
\end{equation}
We know, using \eqref{w_identity} and the fact that $\aaa-\bbb\geq 0$, that
\[- c^2(\aaa-\bbb) \int_0^{t'} (\psi_t,\bbb w_t +\taua  w_{\frakKone t})_{L^2(\Om)}\ds = c^2(\aaa-\bbb) \int_0^{t'} (\psi_t,\bbb \psi +\taua  \frakKone\Lconv\psi_t)_{L^2(\Om)}\ds \geq 0. \]
Thus, for all $t'\in[0,T]$
\[\|\taua \frakKone * \psit + \bbb \psi\|^2_{L^2(\Om)}(t')=0.\]
Using also that $\frakKone$ has a resolvent, we conclude that $\psi = 0$.
\qed
\end{proof}
In this section, we had to assume that $\aaa > \bbb$ to obtain $\tau$-uniform well-posedness of the considered initial-boundary-value problem. We show in what follows the conditions under which we can relax this assumption.
\subsection{Uniform well-posedness for $\aaa=\bbb$ under the stronger assumption $\tfrakKone \in L^2(0,T)$}\label{Sec:a_eq_b}
Theorem~\ref{Prop:Wellposedness_a_g_b} does not provide a uniform well-posedness result in the case $\aaa=\bbb$. We show here that under a stronger assumption on the resolvent $\tfrakKone$, one can obtain uniform-in-$\tau$ well-posedness of the generalized fractional MGT equations, even when $\aaa=\bbb$, in the space
\begin{equation}
	\calX^2 = \{ \psi \in X_{\taua\frakKone}^2(0,T; \Honezero)\,|\, \psit \in X_{\taua\frakKone}^2(0,T; L^2(\Omega))\}.
\end{equation}
In particular, the usage of this space imposes that $\tfrakKone\in L^2(0,T)$ in order to be able to use Lemma~\ref{Lemma:Caputo_seq_compact}. For fractional MGT equations with leading-term kernel/measure $\frakKone = g_{1-\alpha}$, this corresponds to requiring the fractional derivative order $\alpha$ to be larger than $1/2$.
\begin{proposition}\label{Prop:wellposedness_GFE_laws}
	Let $T>0$, $\aaa\geq\bbb>0$. Let $\tfrakKone\in L^2(0,T)$ and let Assumptions~\eqref{eq:boundedness_assumption},  \eqref{eq:causality_eq}--%\eqref{eqn:causality_aaumption},
	\eqref{eqn:k1_positivity_lin}, and \eqref{eqn:positivity} or \eqref{eqn:positivity_nonsing} (see discussion on page~\pageref{discussion_k2_assumption}) hold. Then, given initial data \[(\psi_0,\psi_1,\psi_2^{\frakKone}) \in H_0^1(\Omega) \times \{0\}\times L^2(\Omega)\] and a source term $f\in L^1(0,T;L^2(\Omega))$, 
	there is a unique $\psi \in \calX^2$ which solves 
	\begin{equation}
		\begin{aligned}
			\begin{multlined}[t] -\taua\intT \big(\frakKone * \psitt , v_t\big)_{L^2} \ds - \intT(\aaa\psit, v_t)_{L^2}\ds + c^2  \intT(\bbb \nabla \psi, \nabla v)_{L^2}\ds \\
				+\tau^a c^2 \intT\big(\frakKone \Lconv \nabla \psit , \nabla v\big)_{L^2}\ds - \rt^b \delta  \intT(\frakKtwo \Lconv \nabla \psit, \nabla v_t)_{L^2} \ds
				\\= - \taua \big(\psitwo , v(0)\big)_{L^2} + \intT(f, v)_{L^2} \ds,
			\end{multlined}	
		\end{aligned}		
	\end{equation}
	for all $v\in H^1(0,T;H_0^1(\Omega))$ such that $v(T) = 0$, with
	\[ (\psi,\psit) \Big|_{t=0} = (\psi_0,0).\]
	Furthermore, for almost all $t \in (0,T)$, the solution satisfies 
	\begin{equation}
		\begin{multlined}
			\|\tau^a \frakKone\Lconv\psi_{tt} \|^2_{L^2(0,t;L^2(\Omega))}+
			\|\psi_{t}\|^2_{L^2(0,t;L^2(\Omega))}+  \|\tau^a \frakKone\Lconv\nabla\psi_t\|^2_{L^2(0,t;L^2(\Omega))}+\| \nabla \psi\|^2_{L^2(0,t;L^2(\Omega))}\\ \lesssim_T \tau^{2a}\|\psitwo\|^2_{L^2(\Omega)}+\|\nabla\psi_0\|^2_{L^2(\Omega)}+\|f\|^2_{L^1(0,t;L^2(\Omega))},
		\end{multlined}
	\end{equation} 
	where the hidden constant is independent of $\delta$ and $\tau$.
\end{proposition}
\begin{proof}
	Recall that in the course of the proof of Theorem~\ref{Prop:Wellposedness_a_g_b}, we obtained the following $\tau$-uniform estimate on $z$ given in \eqref{eqn:linear_z_estimate}:
	\begin{equation}
		\begin{multlined}
			\|z_t(t)\|^2_{L^2(\Omega)}+ \frac\bbb\aaa c^2 \|\nabla z (t)\|^2_{L^2(\Omega)} \lesssim \|f\|^2_{L^1(0,t;L^2(\Omega))} +  \tau^{2a}\|\psitwo\|^2_{L^2(\Omega)} + \|\nabla\psi_0\|^2_{L^2(\Omega)}.
		\end{multlined}
	\end{equation}
	The main idea here is to use a different bootstrap argument than the one used in Theorem~\ref{Prop:Wellposedness_a_g_b}.
	To obtain the estimates of $\frakKone\Lconv \psi_t$ and $\psi$ separately, we again consider the auxiliary problem
	\begin{equation}\label{eqn:bootstrap_a_b}
		\tau^a\frakKone \Lconv \psi_t + \aaa\psi = z \quad \textrm{a.e.\ in }\Omega,
	\end{equation}
	with $z \in W^{1,\infty}(0,T;L^2(\Om)) \cap L^\infty(0,T;H^1(\Om))$, $\psi(0)= \psi_0$, and $\psi|_{\partial\Om} = 0$. 
	This time instead of using an explicit variation of constants formula, we test \eqref{eqn:bootstrap_a_b} with $-\Delta\psi$, which is allowed in the semi-discrete setting, to obtain 
	\[
	\|\psi\|_{L^2(0,t;H^1(\Om))} \lesssim \|z\|_{L^2(0,t;H^1(\Om))} \lesssim_T \|z\|_{L^\infty(0,t;H^1(\Om))}.
	\]
	Here, we have again used the positivity of $\frakKone$ stated in \eqref{eqn:k1_positivity_lin}.
	Similarly, using that $\psi_t(0) = 0$, we test the time-differentiated \eqref{eqn:bootstrap_a_b} with $\psi_t$ to obtain 
	\[
	\|\psi_t\|_{L^2(0,t;L^2(\Om))} \lesssim \|z_t\|_{L^2(0,t;L^2(\Om))} \lesssim_T \|z_t\|_{L^\infty(0,t;L^2(\Om))},
	\]
	which gives the desired estimate.
	\subsubsection*{Passing to the limit}
	We also discuss here the subtleties of passing to the limit in this setting where we achieve a lower regularity in time compared to that of Theorem~\ref{Prop:Wellposedness_a_g_b}. In particular, thanks to the stronger assumption $\tfrakKone \in L^2(0,T)$, we can use Lemma~\ref{Lemma:Caputo_seq_compact} to extract weakly convergent subsequences (not relabeled), such that 
	\begin{alignat}{2}
		\psi^{(n)} &\stackrel{}{\relbar\joinrel\rightharpoonup} \psi \quad  &&\textrm{weakly in } L^2(0,T; H_0^1(\Om)),\\
		\psi^{(n)}_t &\stackrel{}{\relbar\joinrel\rightharpoonup} \psi_t \quad  &&\textrm{weakly in } L^2(0,T; L^2(\Om)),\\
		\frakKone \Lconv \psi^{(n)}_t &\stackrel{}{\relbar\joinrel\rightharpoonup} \frakKone \Lconv \psi_t \quad  &&\textrm{weakly in } L^2(0,T; H_0^1(\Om)),\\
		\frakKone \Lconv \psi^{(n)}_{tt} &\stackrel{}{\relbar\joinrel\rightharpoonup} \frakKone \Lconv \psi_{tt} \quad && \textrm{weakly in } L^2(0,T; L^2(\Om)).
		\shortintertext{By $\frakKtwo = \frakR\Lconv \frakKone$ and Young's inequality:}
		\frakKtwo \Lconv \psi^{(n)}_{tt} &\stackrel{}{\relbar\joinrel\rightharpoonup} \frakKtwo \Lconv \psi_{tt}  \quad &&\textrm{weakly in } L^2(0,T; L^2(\Om)),\\
		\frakKtwo \Lconv \psi^{(n)}_{t} &\stackrel{}{\relbar\joinrel\rightharpoonup} \frakKtwo \Lconv \psi_{t} \quad  &&\textrm{weakly in } L^2(0,T; \Honezero). 
	\end{alignat}
	By Lemma~\ref{Lemma:Caputo_compact_embedding}, there is a subsequence (again not relabeled), such that 
	\begin{equation}
		\begin{alignedat}{4} 
			\psi^{(n)} &\longrightarrow \psi && \quad \text{ strongly}  &&\text{ in } &&L^2(0,T; L^2(\Om)),  \\
			\psit^{(n)} &\longrightarrow \psit && \quad \text{ strongly}  &&\text{ in } && L^2(0,T; H^{-1}(\Om)).
		\end{alignedat} 
	\end{equation} 
	This is enough to pass in the weak form to the limit in the semi discrete equation and show that $\psi$ solves:
	\begin{equation}
		\begin{aligned}
			\begin{multlined}[t]- \taua\intT \big(\frakKone * \psitt , v_t\big)_{L^2} \ds-  \intT(\aaa\psit, v_t)_{L^2}\ds + c^2  \intT(\bbb \nabla \psi, \nabla v)_{L^2}\ds \\
				+\tau^a c^2 \intT\big(\frakKone \Lconv \nabla \psit , \nabla v\big)_{L^2}\ds -\rt^b \delta  \intT(\frakKtwo \Lconv \nabla \psit, \nabla v_t)_{L^2} \ds
				\\= - \taua \big(\psitwo , v(0)\big)_{L^2} + \intT(f, v)_{L^2} \ds,
			\end{multlined}	
		\end{aligned}		
	\end{equation}
	for all $v \in H^1(0,T;H_0^1(\Om))$ such that $v(T) = 0$, where we have used $\frakKtwo \Lconv \psitt = (\frakKtwo \Lconv \psit)_t$.
	Note that due to Lemma~\ref{Lemma:Caputo_seq_compact} we have  
	\(\psi \in C([0,T]; H_0^1(\Om))\) and \(\psi_t \in C([0,T]; L^2(\Om)).\)
	The rest of the arguments works similarly to those used for Theorem~\ref{Prop:Wellposedness_a_g_b}. The details are omitted.
	\qed
\end{proof}

\subsection{Removing the restriction $\aaa\geq\bbb$ for $\frakR = 1$}\label{Sec:r_one}
Notice that the previous two Sections~\ref{Sec:a_g_b} and \ref{Sec:a_eq_b} had a restriction on the relation order of $\aaa$ and $\bbb$. Section~\ref{Sec:a_eq_b} additionally introduced a stricter requirement on the regularity of $\tfrakKone$. We can get rid of both of these restrictions in the case $\frakR = 1$ (which is valid, for example, for wave models fMGT II and MGT) in a straightforward manner, owing to the fact that $\psi_1 = 0$ and thus
\[\frakKtwo \Lconv \Delta \psi_{tt} = 1\Lconv \frakKone \Lconv\Delta \psi_{tt} =  \frakKone \Lconv \Delta \psi_{t}.\]
Equation \eqref{eqn:first_eq} then reduces to
\begin{equation}\label{eq:GFEII_z_form_Prop}
	\big(\tau^a\frakKone \Lconv \psi_{tt})_{t} + \aaa\psi_{tt} - \big(c^2  \tau^a + \delta \rt^b\big)\frakKone \Lconv \Delta\psi_t - c^2 \bbb \Delta \psi   = f.
\end{equation}

To justify the importance of $\frakR =1 $, let us rewrite \eqref{eq:GFEII_z_form_Prop} as
\begin{equation}
	\big(\tau^a\frakKone \Lconv \psi_t + \aaa\psi\big)_{tt} - \frac\bbb\aaa c^2 \Delta \big(\tau^a\frakKone \Lconv \psi_t + \aaa\psi \big) - \Big(\frac{\aaa-\bbb}\aaa c^2  \tau^a + \delta \rt^b \Big)\frakKone \Lconv \Delta\psi_t = f.
\end{equation}
Note that for $\delta>0$ fixed,  if $\aaa - \bbb$ is nonnegative or if $\tau \searrow 0$ (i.e. small enough), then the resulting damping coefficient is bounded away from zero, i.e., there exists $\underline \delta > 0$ such that
\begin{equation}\label{tau_bar_cond}
	\Big(\frac{\aaa-\bbb}\aaa c^2  \tau^a + \delta \rt^b \Big) > \underline \delta.
\end{equation}
Equation~\eqref{eq:GFEII_z_form_Prop} behaves then again as a damped wave equation for $z =\tau^a\frakKone \Lconv \psi_t + \aaa\psi$ (similarly to what was argued for \eqref{eq:nonunfiorm_a_geq_b}).

Condition \eqref{tau_bar_cond} above corresponds to requiring strict dissipation in the case of the integer-order MGT equation; see e.g. \cite{kaltenbacher2011wellposedness,kaltenbacher2020vanishing,bongarti2020singular}. Indeed, in that case $\aaa=\bbb=1$ and $b=0$, thus the condition reduces to $\delta > \underline \delta >0$ (which holds by picking, for example, $\underline \delta = \dfrac\delta2$).
\begin{proposition}\label{Cor:wellposedness_GFEII}
		Let $T>0$, $\aaa, \bbb \in  \R_+$, and let $\bar\tau>0$ be small enough such that \eqref{tau_bar_cond} holds for all $ \tau \in (0,\bar\tau]$. Let Assumptions \eqref{eq:boundedness_assumption}, %\eqref{eq:causality_eq},
		\eqref{eqn:resolvent}, and \eqref{eqn:k1_positivity_lin} hold. Then, given initial data \[(\psi_0,\psi_1,\psi_2^{\frakKone}) \in H_0^1(\Omega) \times \{0\}\times L^2(\Omega)\] and a source term $f\in L^1(0,T;L^2(\Omega))$, 
		there is a unique $\psi \in \calX^\infty$ which solves 
		\begin{equation}
			\begin{aligned}
				\begin{multlined}[t] -\taua\intT \big(\frakKone * \psitt , v_t\big)_{L^2} \ds - \intT(\aaa\psit, v_t)_{L^2}\ds + c^2  \intT(\bbb \nabla \psi, \nabla v)_{L^2}\ds \\
					+(\tau^a c^2+\rt^b \delta) \intT\big(\frakKone \Lconv \nabla \psit , \nabla v\big)_{L^2}\ds 
					= - \taua \big(\psitwo , v(0)\big)_{L^2}  + \intT(f, v)_{L^2} \ds,
				\end{multlined}	
			\end{aligned}		
		\end{equation}
		for all $v\in H^1(0,T;H_0^1(\Omega))$ such that $v(T) = 0$, with
		\[ (\psi,\psit) \Big|_{t=0} = (\psi_0,0).\]
		Furthermore, for almost all $t\in(0,T)$, the solution satisfies 
	\begin{equation}
		\begin{multlined}
			\|\tau^a \frakKone\Lconv\psi_{tt} (t)\|^2_{L^{2}(\Omega)}+
			\|\psi_{t}(t)\|^2_{L^2(\Omega)}+  \check C(\tau,\delta)\| \frakKone\Lconv\nabla\psi_t(t)\|^2_{L^2(\Omega)}+\| \nabla \psi(t)\|^2_{L^2(\Omega)}\\ \lesssim \tau^{2a}\|\psitwo\|^2_{L^2(\Omega)}+\|\nabla\psi_0\|^2_{L^2(\Omega)}+\|f\|^2_{L^1(0,T;L^2(\Omega))},
		\end{multlined}
	\end{equation} 
	where the hidden constant is independent of $\tau$ and $T$.
\end{proposition}
The constant $\check C(\tau,\delta) >0$ will be made precise below; see \eqref{eqn:C_tau_delta}.
\begin{proof}
	The proof follows again by a Galerkin procedure. We focus on the energy estimate here and omit other details.
	Testing with $\frakKone\Lconv\psi_{tt}$ in \eqref{eq:GFEII_z_form_Prop}, which is allowed in the semi-discrete setting, yields
	\begin{equation}\label{eqn:GFEII_k1_testing1}
		\begin{multlined}
			\frac12 \tau^a\|\frakKone \Lconv \psi_{tt}(s)\|^2_{L^2(\Om)}\Big|^t_0 + \intt(\aaa \psi_{tt}, \frakKone \Lconv \psi_{tt})_{L^2}\ds + \frac12(\tau^a c^2 + \rt^b\delta ) \|\frakKone \Lconv \nabla\psi_{t}(s)\|^2_{L^2(\Om)}\Big|^t_0 \\+ c^2\intt(\bbb \nabla\psi,\frakKone \Lconv \nabla\psi_{tt})_{L^2}\ds = \intt (f,\frakKone \Lconv \psi_{tt})_{L^2}\ds,
		\end{multlined}\tag{I}
	\end{equation}
	while when we test with $\psi_{t}$, we obtain
	\begin{equation}
		\begin{multlined}\label{eqn:GFEII_k1_testing2}
			\tau^a \intt \big((\frakKone \Lconv \psi_{tt})_t, \psi_t\big)_{L^2} \ds+ \frac\aaa2\|\psi_t (s)\|^2_{L^2(\Om)}\Big|^t_0 + (\tau^a c^2 + \rt^b\delta ) \intt \big(\frakKone \Lconv \nabla\psi_t,\nabla \psi_t\big)_{L^2} \ds \\+ \frac\bbb2c^2\|\nabla\psi (s)\|^2_{L^2(\Om)}\Big|^t_0 = \intt (f,\psi_{t})_{L^2} \ds .
		\end{multlined}\tag{II}
	\end{equation}
	We intend to calculate $\taua$ \eqref{eqn:GFEII_k1_testing1} $+\gamma$ \eqref{eqn:GFEII_k1_testing2} with $\gamma$ verifying 
	\[\frac{{\tau^ac^2}\bbb }{\rt^b \delta+ c^2 \tau^a} + \mu< \gamma < \aaa,\]
	for some $\mu>0$.
	The existence of such $\gamma$ follows from the assumption on $\bar\tau$ with the choice $$\mu= \dfrac{\aaa \underline \delta }{\rt^b \delta + c^2\tau^a}.$$
	Let us first use integration by parts on two of the terms, namely:
		\begin{equation}
		\begin{multlined}
			c^2 \intt(\bbb \nabla\psi,\frakKone \Lconv \nabla\psi_{tt})_{L^2} \ds
			=-c^2 \intt(\bbb \nabla\psi_t ,\frakKone \Lconv \nabla\psi_{t})_{L^2}\ds + c^2(\bbb \nabla\psi(t),\frakKone \Lconv \nabla\psi_{t}(t))_{L^2},
		\end{multlined}
	\end{equation}	
	because, as argued before, $(\frakKone \Lconv \psit) (0)= 0$. Similarly, recalling that $\psi_t(0)=0$, we have
	\begin{equation}
		\begin{multlined}
			\tau^a \intt \big((\frakKone \Lconv \psi_{tt})_t, \psi_t \big)_{L^2} \ds
			= - \intt\tau^a \big(\frakKone \Lconv \psi_{tt}, \psi_{tt}\big)_{L^2} \ds + \tau^a (\frakKone \Lconv \psi_{tt}(t), \psi_{t}(t))_{L^2}.
		\end{multlined}
	\end{equation}
	Secondly, we use Young's inequality on the terms:
	\begin{equation}
		\begin{multlined}
			\taua c^2\bbb (\nabla\psi(t),\frakKone \Lconv \nabla\psi_{t}(t))_{L^2} \leq  \frac{c^2 \bbb (\gamma -\mu) }{2} \|\nabla \psi(t)\|^2_{L^2} + \tau^{2a}\frac{c^2 \bbb}{2(\gamma-\mu)}\|\frakKone \Lconv \nabla\psi_{t}(t)\|^2_{L^2}.
		\end{multlined}
	\end{equation}
	and
	\begin{equation}
	\begin{multlined}
		\gamma (\taua \frakKone \Lconv \psi_{tt}(t), \psi_{t}(t))_{(L^2)} \leq \frac1{2} \|\taua\frakKone \Lconv \psi_{tt}(t)\|^2_{L^2} + \frac12 \gamma ^2\|\psi_t(t)\|^2_{L^2}.
	\end{multlined}
\end{equation}
	Together, after calculating $\taua$ \eqref{eqn:GFEII_k1_testing1} $+\gamma$ \eqref{eqn:GFEII_k1_testing2}, we obtain 
\begin{equation}
	\begin{multlined}
		\|\tau^a \frakKone \Lconv \psi_{tt}(s)\|^2_{L^2}\Big|^t_0 +  \tau^a(\aaa-\gamma) \intt \big(\psi_{tt}, \frakKone \Lconv \psi_{tt}\big)_{L^2} \ds 
	 	+  \check C(\tau,\delta) \|\frakKone \Lconv \nabla\psi_{t}(s)\|^2_{L^2}\Big|^t_0 \\+ (-c^2\taua \bbb+\gamma  c^2 \tau^a  +\gamma\rt^b \delta) \intt (\nabla\psi_t,\frakKone \Lconv \nabla\psi_{t})_{L^2} \ds + \frac{\gamma(\aaa-\gamma)}2\|\psi_t(s)\|^2_{L^2}\Big|^t_0  \\+ \frac{c^2\bbb}2 \mu\|\nabla\psi(s)\|^2_{L^2}\Big|^t_0
		\leq \intt (f,\taua \frakKone \Lconv \psi_{tt} + \gamma \psit)_{L^2}\ds,
	\end{multlined}
\end{equation}
	where 
	\begin{equation}\label{eqn:C_tau_delta} 
		\check C(\tau,\delta) = \frac\taua2\Big(\tau^a c^2 + \rt^b\delta - \frac{\tau^a c^2 \bbb}{\gamma-\mu} \Big).
	\end{equation}
	Noticing that $-c^2 \tau^a \bbb+\gamma  c^2\tau^a  +\gamma\rt^b \delta>0$, $\aaa-\gamma>0$, and $\check C(\tau,\delta)>0$,
	we conclude that 
	\begin{equation}
		\begin{multlined}
			\|\tau^a \frakKone\Lconv\psi_{tt} (t)\|^2_{L^{2}(\Omega)}+
			\|\psi_{t}(t)\|^2_{L^2(\Omega)}+  \check C(\tau,\delta) \|\frakKone\Lconv\nabla\psi_t(t)\|^2_{L^2(\Omega)}+\| \nabla \psi(t)\|^2_{L^2(\Omega)}\\ \lesssim \tau^{2a}\|\psitwo\|^2_{L^2(\Omega)}+\|\nabla\psi_0\|^2_{L^2(\Omega)}+\|f\|^2_{L^1(0,t;L^2(\Omega))},
		\end{multlined}
	\end{equation} 
	where we have used the positivity assumption on $\frakKone$ stated in \eqref{eqn:k1_positivity_lin}.
	The hidden constant is independent of $\tau$ and $T$.
	\qed
\end{proof}

\begin{remark}[On the generalized fractional MGT with smooth variable-in-time coefficients]
To be able to show well-posedness of the equations with a variable coefficient $\aaa =\aaa(t)$ following the strategy given above, one would need to control the following term arising from testing with $\frakKone \Lconv \psitt$:
\[\intT(\frakKone \Lconv \psitt, \aaa \psitt)_{L^2} \ds.\]
In the integer order case we are saved by the fact that, in the worst case scenario, we can transfer some of the differentiation to $\aaa$, using 
\[2\aaa \psi_{tt} \psi_t = (\aaa \psi_t^2)_t - \aaa_t \psit^2.\] 
This essentially relies on the Leibniz formula for integer-order differentiation. Looking at \cite[Theorems 3.17, 2.18]{diethelm2002analysis}, one sees that the infinite series equivalent for fractional derivatives is particularly unwieldy.

Obviously other approaches have been devised to show coercivity of fractional derivatives (see, e.g., \cite[Lemma 1]{alikhanov2010priori}). However, it is not straightforward to adapt to the present context the proof of the cited lemma or those of \cite[Sections 18.4, 18.5]{gripenberg1990volterra}, which rely on the absolute continuity or boundedness of the kernel. However, if $\aaa_t\leq0$ then we can get rid of the smoothness assumption on the kernel using the trick of smooth approximations (similarly to, e.g., \cite[lemma B.1]{kaltenbacher2021determining}), but in the general case we could not obtain an expression that would be suitable to work with.

Nevertheless, it is possible to exploit specific structures of particular fractional equations to devise different testing strategies and show well-posedness of variable-coefficient counterparts of \eqref{eqn:first_eq}; see~\cite{nikolic2023nonlinear} where the case $\frakKtwo =1$ is treated. The analysis there covers, for example, nonlinear wave equations based on heat flux law GFE.
\end{remark}
\section{Vanishing relaxation time limit}\label{Sec:vanishing_time_limit}
In this section, we aim to establish convergence results for the generalized fractional Moore--Gibson--Thompson equations for the vanishing relaxation time limit. We assume throughout that the initial data do not depend on $\tau$. Relaxing this condition may be achieved by requiring $\big(\psi_0^{(\tau)}, \psi_1^{(\tau)}\big)$ to converge strongly in $H^1(\Omega)\times \{0\}$ (with a suitable rate of convergence) and $\psi_2^{\frakKone,(\tau)}$ to be uniformly bounded in $L^2(\Om)$.

Below we consider a sequence \[\Big(\psi^{(\tau)}\Big)_{\tau \in(0,\overline \tau]},\] 
where the upper bound $\overline\tau$ can be any positive real constant in the cases of Theorem~\ref{Prop:Wellposedness_a_g_b} and Proposition~\ref{Prop:wellposedness_GFE_laws}. When in the setting of Proposition~\ref{Cor:wellposedness_GFEII}, we require $\overline\tau$ to be small enough so that \eqref{tau_bar_cond} holds for all $ \tau \in (0,\bar\tau]$.

In this section, we allow the following cases:
\begin{itemize}
	\item[a)] $\aaa >\bbb$, for which uniform-in-$\tau$ well-posedness was established in Theorem~\ref{Prop:Wellposedness_a_g_b};
	\item[b)] $\aaa \geq \bbb$ and $\tfrakKone \in L^2(0,T)$, whose $\tau$-uniform well-posedness has been the subject of Proposition~\ref{Prop:wellposedness_GFE_laws};
	\item[c)] $\aaa,\,\bbb>0$ and $\frakR = 1$, studied in
	 Proposition~\ref{Cor:wellposedness_GFEII}.
\end{itemize}
Below, we will not distinguish the cases a) and c) as their treatment is similar. The case b) will however be treated separately since the spaces of convergence to a vanishing relaxation limit will be different.  

\subsection{Weak vanishing relaxation time limits}
From the well-posedness analysis, we have established uniform-in-$\tau$ bounds on $\psi$ and $\psi_t$. In this section we use these to prove weak convergence of $\psi^{(\tau)}$ as $\tau \searrow 0$ in suitable spaces.
The limiting equation is 	 given formally by 
\begin{equation}\label{eq:limiting_eq}
	\aaa\psi_{tt} - c^2 \bbb \Delta \psi -   \delta \rt^b \frakKtwo \Lconv \Delta\psi_{tt}   = f.
\end{equation}
This equation will hold in the following weak sense:
\begin{equation}\label{eq:weak_limiting_eq}
	\begin{aligned}
		\begin{multlined}[t]
		-  \intT(\aaa\psit^{(0)}, v_t)_{L^2}\ds + c^2  \intT(\bbb \nabla \psi^{(0)}, \nabla v)_{L^2}\ds  +\rt^b \delta  \intT(\frakKtwo \Lconv \nabla \psi^{(0)}, \nabla v_{tt})_{L^2} \ds\\
			= - \rt^b \delta \intT (\frakKtwo \nabla \psi_0, \nabla v_t)_{L^2} \ds +\intT(f, v)_{L^2} \ds, 
		\end{multlined}	
	\end{aligned}		
\end{equation}
with initial data $\psi^{(0)}= \psi_0$, for all $v \in H^2(0,T;H_0^1(\Om))$ such that $v(T)=v_t(T)=0$.
\\[5mm]
\noindent{\it Case a) and c).}
\begin{proposition}\label{Prop:weak_conv_1}
	Let the assumptions of Theorem~\ref{Prop:Wellposedness_a_g_b} (if $\aaa > \bbb$) or of Proposition~\ref{Cor:wellposedness_GFEII} (if $\aaa,\bbb>0$ and $\frakR =1$) hold. Then a sequence $\Big(\psi^{(\tau)}\Big)_{\tau \in(0,\overline \tau]}$ of solutions to \eqref{eqn:ibvp_eqn} converges weakly-$*$ to the solution $ \psi^{(0)} \in L^\infty(0,T; H_0^1(\Om)) \cap W^{1,\infty}(0,T; L^2(\Om))$ of \eqref{eq:weak_limiting_eq} as $\tau \searrow 0$.
\end{proposition}
\begin{proof}
	From the energy estimate of Theorem~\ref{Prop:Wellposedness_a_g_b} (or of Proposition~\ref{Cor:wellposedness_GFEII}), we know that the following quantities are bounded uniformly in $\tau$:
\begin{align}
	\psi^{(\tau)} &
	\quad \textrm{bounded in } L^\infty(0,T; H_0^1(\Om)),\\
	\psi^{(\tau)}_t &
	\quad \textrm{bounded in } L^\infty(0,T; L^2(\Om)). 
\end{align}
This implies the existence of a subsequence (not relabeled) such that
\begin{alignat}{2}
	\psi^{(\tau)} &\stackrel{}{\relbar\joinrel\rightharpoonup} \psi 
	\qquad && \textrm{weakly-$*$ in } L^\infty(0,T; H_0^1(\Om)),\\
	\psi^{(\tau)}_t &\stackrel{}{\relbar\joinrel\rightharpoonup} \psi_t 
	&& \textrm{weakly-$*$ in } L^\infty(0,T; L^2(\Om)). 
\end{alignat}
In particular, thanks to the Aubin-Lions-Simon Lemma (see \cite[Corollary 4]{simon1986compact}), we have that 
\begin{align}\label{eq:strong_conv}
	\psi^{(\tau)} &\stackrel{}{\relbar\joinrel\rightarrow} \psi 
	\quad \textrm{strongly in } C([0,T]; L^2(\Om)). 
\end{align}
Thus the initial data $\psi^{(0)}(0) = \psi_0$ is attained.
Moreover, by boundedness of the operator $T_{\frakKtwo}$, $\frakKtwo \Lconv \psi^{(\tau)}$ is uniformly bounded in $L^\infty(0,T; \Honezero)$, thus
\begin{align}
\frakKtwo \Lconv \psi^{(\tau)} &\stackrel{}{\relbar\joinrel\rightharpoonup} \frakKtwo \Lconv \psi \quad \textrm{weakly-$*$ in } L^\infty(0,T; \Honezero). 
\end{align}

We would like to use the established weak convergences to go to the limit in the weak form
	\begin{equation}
	\begin{aligned}
		\begin{multlined}[t]- \taua\intT \big(\frakKone * \psitt^{(\tau)} , v_t\big)_{L^2} \ds-  \intT(\aaa\psit^{(\tau)}, v_t)_{L^2}\ds + c^2  \intT(\bbb \nabla \psi^{(\tau)}, \nabla v)_{L^2}\ds \\
			+\tau^a c^2 \intT\big(\frakKone \Lconv \nabla \psit^{(\tau)} , \nabla v\big)_{L^2}\ds +\rt^b \delta  \intT(\frakKtwo \Lconv \nabla \psi^{(\tau)}, \nabla v_{tt})_{L^2} \ds
			\\= - \taua \big(\psitwo , v(0)\big)_{L^2} - \rt^b \delta \intT (\frakKtwo \nabla \psi_0, \nabla v_t)_{L^2} \ds  + \intT(f, v)_{L^2} \ds,
		\end{multlined}	
	\end{aligned}		
\end{equation}
for all $v \in H^2(0,T;H_0^1(\Om))$ with $v(T) = v_t(T) = 0$. In particular, to show the desired limiting behavior, it suffices to show that 
the $\tau$-weighted terms on the left-hand side above converge weakly to 0. We borrow an idea given in \cite{kaltenbacher2023vanishing}, where integration by parts is used to show the sought-after property. To this end, recall that in our setting
\[\frakKone * \psitt^{(\tau)} = (\frakKone * \psit^{(\tau)})_t,\]
since $\psit^{(\tau)}(0) = \psi_1 = 0$. 
Therefore
\begin{equation}
	\begin{aligned}
		\begin{multlined}
			\taua\intT \big(\frakKone * \psitt^{(\tau)} , v_t\big)_{L^2} \ds =  \taua\intT \big((\frakKone * \psit^{(\tau)})_t , v_{t}\big)_{L^2} \ds 
			\\
			=-\taua\intT \big(\frakKone * \psit^{(\tau)} , v_{tt}\big)_{L^2} \ds 
			\ {\relbar\joinrel\rightarrow} \ 0,
		\end{multlined}
	\end{aligned}
\end{equation}
where we have used the uniform boundedness in $\tau$ of $\|\psi_t^{(\tau)}\|_{L^2(0,T;L^2(\Omega))}$. 
To treat the fourth term on the left-hand side, we introduce the time reversed variable $\widetilde v (s) = v(T-s)$ for all $s \in (0,T)$
to write 
\begin{equation}
	\begin{multlined}
		\intT(\frakKone \Lconv \nabla\psit^{(\tau)}, \nabla v)_{L^2} \ds 
		= \intO \frakKone \Lconv \nabla  \psit^{(\tau)} \Lconv \nabla \widetilde v \dx 
		= \intT \big(\nabla\psit^{(\tau)}(s) , (\frakKone \Lconv \nabla \widetilde v_{t})(T-s)\big)_{L^2} \ds.
	\end{multlined}
\end{equation}
Using that $\widetilde v_t(0) = - v_t(T) = 0$, we integrate by parts and write
\begin{equation}
	\begin{aligned}
		\begin{multlined}
			\taua \intT\big(\frakKone \Lconv \nabla \psit^{(\tau)} , \nabla v\big)_{L^2}\ds =  \taua\intT \big(\nabla\psit^{(\tau)}(s) , (\frakKone \Lconv \nabla \widetilde v_{t})(T-s)\big)_{L^2} \ds \\
			=  \taua\intT \big(\nabla\psi^{(\tau)}(s) , (\frakKone \Lconv \nabla \widetilde v_{tt})(T-s)\big)_{L^2} \ds + \taua \big(\nabla\psi_0 , (\frakKone \Lconv \nabla \widetilde v_{tt})(T)\big)_{L^2}
			\ {\relbar\joinrel\rightarrow} \ 0,
		\end{multlined}
	\end{aligned}
\end{equation}
as $\tau \searrow 0$.
\\
\indent Uniqueness of the limit is ensured by seeing that the difference of two solutions of the limiting problem, $\bar \psi = \psi^{(1)} - \psi^{(2)}$, would have to solve 
 \begin{equation}\label{eq:uniquenes_weak_limit}
 	\begin{aligned}
 		\begin{multlined}[t]
 		-  \intT(\aaa\bar\psi_t, v_t)_{L^2}\ds + c^2  \intT(\bbb \nabla \bar\psi, \nabla v)_{L^2}\ds 
 		 +\rt^b \delta  \intT(\frakKtwo \Lconv \nabla \bar\psi, \nabla v_{tt})_{L^2} \ds
 			= 0,
 		\end{multlined}	
 	\end{aligned}		
 \end{equation}
with zero initial data, for all $v\in H^2(0,T;H^1_0(\Omega))$ such that $v(T)=v_t(T)=0$. If we were allowed to test with $\bar\psit$ the uniqueness would be immediate. However, $\bar\psit$ lacks regularity in our setting to be a valid test function. We devise an alternative testing strategy to accomodate the low-regularity setting. Let
\begin{equation}
	w (t)=
	\left\{\begin{array}{ll}
		\int_t^{t'} \int_s^{t'} 1\Lconv\bar\psi(\zeta) \, \textup{d} \zeta \ds  \quad  & \textrm{ if} \ \ 0\leq t\leq t',\\
		0 \quad  & \textrm{ if} \ \ t'\leq t\leq T.
	\end{array}\right.
\end{equation}
Thus we have $w(t') = w_t(t')= w_{tt}(0) =0$, and $w_{ttt} = \bar\psi$. Integrating by parts in \eqref{eq:uniquenes_weak_limit} yields 
 \begin{equation}
	\begin{aligned}
		\begin{multlined}[t]
			\int^{t'}_0(\aaa w_{ttt}, w_{tt})_{L^2}\ds + c^2  \int^{t'}_0(\bbb \nabla w_{tt}, \nabla w_t)_{L^2}\ds 
			+\rt^b \delta  \int^{t'}_0(\frakKtwo \Lconv \nabla w_{ttt}, \nabla w_{tt})_{L^2} \ds
			= 0.
		\end{multlined}	
	\end{aligned}		
\end{equation}
Due to the coercivity of $\frakKtwo$ (see \eqref{eqn:positivity} or \eqref{eqn:positivity_nonsing}), we obtain that $w_{tt} = 1 \Lconv \bar\psi= 0$ a.e., which ensures uniqueness. 
By a subsequence-subsequence argument, the whole
sequence $\Big(\psi^{(\tau)}\Big)_{\tau \in(0,\overline \tau]}$ converges to the same limit $\psi^{(0)}$.
\qed
\color{black}
\end{proof}

\noindent{\it Case b).}\\[2mm]
An analogous statement can be made for the case $\aaa = \bbb$ with general heat flux laws. In particular, the following uniform-in-$\tau$ result holds for all fractional MGT equations appearing in Table~\ref{table:r_kernels} (with eventual restrictions on $\alpha$; see Table~\ref{table:r_kernels}) and remains valid for the linear models of \cite{kaltenbacher2022time} with the previously discussed modification on the meaning of the leading fractional derivative (or by setting $\psitt(0) = 0$; see discussion on page \pageref{discussion:RL_leading_term}).

\begin{proposition}\label{Prop:weak_conv_2}
	Let the assumptions of Proposition~\ref{Prop:Wellposedness_a_g_b} hold. Then a family $\Big(\psi^{(\tau)}\Big)_{\tau \in(0,\overline \tau]}$ of solutions to \eqref{eqn:ibvp_eqn}, converges weakly in $L^2(0,T; H_0^1(\Om)) \cap H^1(0,T; L^2(\Om))$ to the solution of \eqref{eq:weak_limiting_eq} as $\tau \searrow 0$.
\end{proposition}
Propositions~\ref{Prop:weak_conv_1} and \ref{Prop:weak_conv_2} ensure convergence to a weak limit. As a byproduct, they establish that equation \eqref{eq:weak_limiting_eq} has a solution in the following space:
\begin{itemize}
\item[\textbullet] $L^\infty(0,T; H_0^1(\Om)) \cap W^{1,\infty}(0,T; L^2(\Om))$ in the case of Proposition~\ref{Prop:weak_conv_1} (i.e., if $\aaa>\bbb$, or if $\aaa,\bbb>0$ with $\frakKtwo$ being expressible as $\frakKtwo = \frakKone\Lconv 1$, with $\frakKone$ verifying assumptions~\eqref{eq:boundedness_assumption}, \eqref{eqn:resolvent}, and \eqref{eqn:k1_positivity_lin})
\item[\textbullet] $L^2(0,T; H_0^1(\Om)) \cap H^1(0,T; L^2(\Om))$ in that of Proposition~\ref{Prop:weak_conv_2} (i.e, when $\aaa\geq \bbb$ and $\frakKtwo\in L^1(0,T)$). 
\end{itemize}
Furthermore, the solution is unique, and we have continuous dependence of the solution on the initial data. To prove the latter, we use the weak lower semi-continuity of the norms to take the $\tau$ limit in the previously established stability estimate (see Theorem~\ref{Prop:Wellposedness_a_g_b}, Proposition~\ref{Prop:wellposedness_GFE_laws}, and 
Proposition~\ref{Cor:wellposedness_GFEII}). Thus, initial-boundary-value problem \eqref{eq:weak_limiting_eq} is well-posed.

Below, we establish the rate of convergence to the identified weak limit.
\subsection{Convergence rates for the vanishing relaxation time}
In the course of the proof of Proposition~\ref{Prop:weak_conv_1}, we obtained also strong convergence of $\Big(\psi^{(\tau)}\Big)_{\tau \in(0,\overline \tau]}$ in $C([0,T];L^2(\Om))$; see~\eqref{eq:strong_conv}.
In what follows, we establish its rate. 
Recall that $\psi^{(0)}$ refers to the solution of the limiting problem given in \eqref{eq:weak_limiting_eq}.
\\[5mm]
\noindent{\it Case a) and c).}
\begin{theorem}\label{Prop:Cauchy_seq1}
	Let the assumptions of Theorem~\ref{Prop:Wellposedness_a_g_b} (if $\aaa > \bbb$) or of Proposition~\ref{Cor:wellposedness_GFEII} (if $\aaa,\bbb>0$ and $\frakR =1$) hold. Then the family $\Big(\psi^{(\tau)}\Big)_{\tau \in(0,\overline \tau]}$ of solutions to \eqref{eqn:ibvp_eqn} converges to $\psi^{(0)}$ in the following sense:
	\[\|\psi^{(\tau)}- \psi^{(0)}\|_{L^\infty(0,T;L^2(\Omega))} = O(\sqrt{\tau^a}) \qquad \textrm{as } \tau \searrow 0.\]
\end{theorem}
\begin{proof}
	Let $t'\in (0,T)$. Similarly to the idea used to show uniqueness in Theorem~\ref{Prop:Wellposedness_a_g_b}, we intend to introduce a test function of the form
	\begin{equation}
		w (t)=
		\left\{\begin{array}{ll}
			\int_t^{t'} \bar\psi(s) \ds \quad  & \textrm{ if} \ \ 0\leq t\leq t',\\
			0 \quad  & \textrm{ if} \ \ t'\leq t\leq T.
		\end{array}\right.
	\end{equation}
	having in mind that $ w_t = -\bar\psi$ and $ w_{tt} = -\bar\psit$.
	The proof of convergence will follow then by showing that $\Big(\psi^{(\tau)}\Big)_{\tau \in(0,\overline \tau]}$ is a Cauchy sequence with respect to the norm $\|\cdot\|_{L^\infty(0,T;L^2(\Omega))}$. We have to do so because the weak form of the limiting equation (see \eqref{eq:weak_limiting_eq}) is too weak to allow us to consider directly the difference $\psi^{(\tau)}-\psi^{(0)}$ when testing with $w$.
	
	Let then $\tau_1,\tau_2>0$ and let $\bar{\psi}=\psi^{(\tau_1)}-\psi^{(\tau_2)}$. Since each of $\psi^{(\tau_1)}$ and $\psi^{(\tau_2)}$ are uniformly-in-$\tau$ bounded in $L^\infty(0,T; H_0^1(\Om)) \cap W^{1,\infty}(0,T; L^2(\Om))$, then so is $\bar\psi$.

Then, $\bar \psi$ satisfies:
\begin{equation}\label{eqn:Cauchy_testing}
	\begin{aligned}
		\begin{multlined}[t]
			- \tau_1^a\int^{t'}_0 \big(\frakKone * \bar{\psi}_{tt} , w_{t}\big)_{L^2} \ds
			-  \int^{t'}_0(\aaa\bar{\psi}_{t}, w_t)_{L^2}\ds + c^2  \int^{t'}_0(\bbb \nabla \bar\psi, \nabla w)_{L^2}\ds 
			\\-\tau_1^a c^2 \int^{t'}_0\big(\frakKone \Lconv \nabla\bar{\psi}, \nabla w_t\big)_{L^2}\ds 
			-\rt^b \delta  \int^{t'}_0(\frakKtwo \Lconv \nabla \bar{\psi}_{t}, \nabla w_t)_{L^2} \ds\\
			= (\taua_1 -\taua_2 )\int^{t'}_0(\frakKone \Lconv \psi^{(\tau_2)}_{tt}, w_t)_{L^2} \ds + (\taua_2 -\taua_1 )\int^{t'}_0(\frakKone  \Lconv \nabla\psit^{(\tau_2)}, \nabla w)_{L^2} \ds\\
			+( \taua_2 - \taua_1) \big(\psitwo , w(0)\big)_{L^2} ,
		\end{multlined}	
	\end{aligned}		
\end{equation}
where we have used that $\bar\psi(0) =  w(t') = 0$ to integrate by parts the fourth term on left-hand side.

First, notice that thanks to assumptions \eqref{eqn:k1_positivity_lin} and \eqref{eqn:positivity} (or, alternatively, \eqref{eqn:positivity_nonsing}), we obtain that 
\[
-\taua_1 c^2 \int^{t'}_0\big(\frakKone \Lconv \nabla\bar{\psi}, \nabla w_t\big)_{L^2}\ds -\rt^b \delta  \int^{t'}_0(\frakKtwo \Lconv \nabla \bar{\psi}_{t}, \nabla w_t)_{L^2} \ds \geq 0,
\]
where we used that $\nabla w_t = -\nabla \bar\psi$.

We need now to treat a few different terms separately.
Using \eqref{eq:boundedness_assumption}, we bound the term
\begin{equation}
	\begin{multlined}
	\Big|- \taua_1\int^{t'}_0 \big(\frakKone * \bar{\psi}_{tt} , w_{t}\big)_{L^2} \ds \Big| = \taua_1\Big|\int^{t'}_0 \big((\frakKone * \bar{\psi}_{t}) , w_{tt}\big)_{L^2} \ds -  \big((\frakKone * \bar{\psi}_{t})(t') , w_{t}(t')\big)_{L^2}\Big|
	\\
	\leq  \taua_1 \|\frakKone\|_{\mm(0,t')}\Big(\|\bar\psi_t\|^2_{L^2(0,t';L^2(\Om))} + 
	\|\bar\psi_t\|_{L^\infty(0,t';L^2(\Om))}
	\|\bar\psi\|_{L^\infty(0,t';L^2(\Om))}\Big),
	\end{multlined}
\end{equation}
where we have used again that since $\bar\psi_t(0) =0$, then $ \frakKone \Lconv \bar \psi_t(0) = 0$. 

Secondly, integrating by parts, we have
\begin{equation}
	\begin{multlined}
	\int^{t'}_0(\frakKone \Lconv \psi^{(\tau_2)}_{tt}, w_t)_{L^2} \ds
	=  \int^{t'}_0(\frakKone \Lconv \psi^{(\tau_2)}_{t}, w_{tt})_{L^2}\ds  - \big((\frakKone * \psi^{(\tau_2)}_{t})(t') , w_{t}(t')\big)_{L^2}.
	\end{multlined}
\end{equation}
We can thus bound this term as follows:
\begin{equation}\label{eq:why_psit_uniform_bound_needed}
	\begin{aligned}
	& \Big|\,(\taua_1 -\taua_2 ) \int^{t'}_0(\frakKone \Lconv \psi^{(\tau_2)}_{tt}, w_t)_{L^2} \ds \,\Big|\\
	\leq\, & |\taua_1 -\taua_2 | \|\frakKone\|_{\mm(0,t')}\Big(\|\psi_t^{(\tau_2)}\|_{L^2(0,t';L^2(\Om))}\|\bar\psi_t\|_{L^2(0,t';L^2(\Om))} \\ &+ \|\psi_t^{(\tau_2)}\|_{L^\infty(0,t';L^2(\Om))}
	\|\bar\psi\|_{L^\infty(0,t';L^2(\Om))}  + \|\psitwo\|_{L^2(\Om)} \|\bar\psi\|_{L^1(0,t';L^2(\Om))}\Big).
	\end{aligned}
\end{equation}
\indent In higher-regularity settings in terms of data, it is conceivable that one may have a $\tau$-uniform bound on $\|\frakKone\Lconv\psi^{(\tau)}_{tt}\|_{L^1(0,t';L^2(\Omega))}$. In such a case, one may pursue showing a higher convergence rate. In our case, due to the low-regularity setting, we needed to integrate by parts and some of the ensuing terms (e.g., $\|\bar\psi_t\|_{L^2(0,t';L^2(\Om))}$) can no longer be absorbed by the left-hand side. Instead we will bound them uniformly by resorting to the triangle inequality and the energy estimates established in Theorem~\ref{Prop:Wellposedness_a_g_b}, or, alternatively, Proposition~\ref{Cor:wellposedness_GFEII}:
\[\|\bar\psi_t\|_{L^2(0,t';L^2(\Om))} \leq \|\psit^{(\tau_1)}\|_{L^2(0,t';L^2(\Om))} + \|\psit^{(\tau_2)}\|_{L^2(0,t';L^2(\Om))}.\]
To treat the second right-hand-side term in \eqref{eqn:Cauchy_testing}, we use again a time-reversed variable $\widetilde w(s) = w (t'-s)$ defined for all $s\in(0,t')$ to write 
\begin{equation}
	\begin{multlined}
		\int^{t'}_0(\frakKone \Lconv \nabla\psit^{(\tau_2)}, \nabla w)_{L^2} \ds 
		= \intO \frakKone \Lconv \nabla  \psit^{(\tau_2)} \Lconv \nabla \widetilde w \dx 
		= \int^{t'}_0\big((\frakKone \Lconv \nabla \widetilde w)(t'-s) , \nabla \psi^{(\tau_2)}_t(s)\big)_{L^2} \ds.
	\end{multlined}
\end{equation}
This form is handier for an integration by parts, especially because $\widetilde w(0) = w (t') = 0,$ such that $$(\frakKone \Lconv \nabla \widetilde w)_t = (\frakKone \Lconv \nabla \widetilde w_t).$$ This yields that
\begin{equation}
	\begin{multlined}
		\int^{t'}_0(\frakKone \Lconv \nabla\psit^{(\tau_2)}, \nabla w)_{L^2} \ds 
		=\int^{t'}_0\big((\frakKone \Lconv \nabla \widetilde w_t)(t'-s) , \nabla \psi^{(\tau_2)}(s)\big)_{L^2} \ds -\big((\frakKone \Lconv \nabla \widetilde w)(t') , \nabla \psi_0\big)_{L^2}.
	\end{multlined}
\end{equation}
Thus, we can bound the term 
\begin{equation}
	\begin{aligned}
		\Big|\,( \taua_2 - \taua_1) \int^{t'}_0(\frakKone \Lconv \nabla\psit^{(\tau_2)},& \nabla w)_{L^2} \ds\,\Big| 
		\\\leq \; &|\taua_1 -\taua_2 | \|\frakKone\|_{\mm(0,t')} \Big(\|\nabla \bar \psi\|_{L^2(0,t';L^2(\Om))}\|\nabla \psi^{(\tau_2)}\|_{L^2(0,t';L^2(\Om))} \\&+ \|\nabla w \|_{L^\infty(0,t';L^2(\Om))}\|\nabla \psi_0\|_{L^2(\Om))}\Big).
	\end{aligned}
\end{equation}
\color{black}
Using the above inequalities in \eqref{eqn:Cauchy_testing} (with $|\taua_1 -\taua_2 |\leq \taua_1 + \taua_2$), one obtains that
\begin{equation}
	\aaa \|\bar\psi(t')\|^2_{L^2(\Om)} + c^2 \bbb \|\nabla w(0)\|^2_{L^2(\Om)} \leq (\taua_1 +\taua_2)\, C_{\frakKone, \psi^{(\tau_1)}, \psi^{(\tau_2)},T},
\end{equation}
with $C_{\frakKone, \psi^{(\tau_1)}, \psi^{(\tau_2)},T}$ being uniformly bounded in $\tau$ thanks to Theorem~\ref{Prop:Wellposedness_a_g_b} or, alternatively, Proposition~\ref{Cor:wellposedness_GFEII}. Thus $\Big(\psi^{(\tau)}\Big)_{\tau \in(0,\overline \tau]}$ is a Cauchy sequence and converges to some limit in $L^\infty(0,T;L^2(\Omega))$. Uniqueness of limits ensures the desired result. Indeed, it suffices to write
\[ \|\psi^{(\tau_1)}(t') - \psi^{(0)}(t')\|_{L^2} \leq \|\bar\psi(t')\|_{L^2} + \|\psi^{(\tau_2)}(t') - \psi^{(0)}(t')\|_{L^2}\]
and to take the limit $\tau_2 \searrow 0$ to get the rate of convergence.
\qed
\end{proof}
The proof above fails for the case $\aaa = \bbb$ where, for a general $\frakR$, we only obtained uniform-in-$\tau$ well-posedness in the space $\calX^2$. Thus a uniform bound on $\|\bar\psi_t\|_{L^\infty(0,t';L^2(\Om))}$ is unavailable (see, e.g., inequality~\eqref{eq:why_psit_uniform_bound_needed}).
In what follows, we show that for general fractional MGT equations with $\aaa= \bbb$ we still obtain strong convergence but in a weaker norm with respect to time. Since the result below has a higher convergence rate than that of Theorem~\ref{Prop:Cauchy_seq1} we will state it for all the cases of interest in this section.
\\[4mm]
\noindent{\it Case a), b), and c).}
\begin{theorem}\label{Prop:Cauchy_seq2}
	Let the assumptions  of Theorem~\ref{Prop:Wellposedness_a_g_b} (if $\aaa > \bbb$), or of Proposition~\ref{Prop:wellposedness_GFE_laws} (if $\aaa \geq \bbb$),  or of Proposition~\ref{Cor:wellposedness_GFEII} (if $\aaa,\bbb>0$ and $\frakR =1$) hold. Then
	\[\|1\Lconv\psi^{(\tau)}- 1\Lconv\psi^{(0)}\|_{L^\infty(0,T;L^2(\Omega))} = O(\tau^a) \qquad \textrm{as } \tau \searrow 0.\]
\end{theorem}
\begin{proof}
The main idea of the proof is to get rid of the necessity of a uniform bound on $\|\psi_t\|_{L^\infty(0,T;L^2(\Omega))}$ which we needed for Theorem~\ref{Prop:Cauchy_seq1}. To this end, similarly to the proof of uniqueness in Proposition~\ref{Prop:weak_conv_1}, we use the following test function:
\begin{equation}
	w (t)=
	\left\{\begin{array}{ll}
		\int_t^{t'} \int_s^{t'} 1\Lconv\bar\psi(\zeta) \, \textup{d} \zeta \ds  \quad  & \textrm{ if} \ \ 0\leq t\leq t',\\
		0 \quad  & \textrm{ if} \ \ t'\leq t\leq T.
	\end{array}\right.
\end{equation}
Now, not only $w(t') =0$, but also $w_t(t')= w_{tt}(0) =0 $. Moreover $w_{ttt} = \bar\psi$.
Note that the regularity of this test-function allows us to consider now directly the difference $\bar \psi = \psi^{(\tau)} - \psi^{(0)}$ for some $\tau \in (0,\bar\tau]$, and a Cauchy-sequence argument is not needed.

Let $\tau \in (0,\bar\tau]$. Recall that $\bar\psi(0) = \bar \psit(0) = 0$, it then follows that
\begin{equation}\label{eqn:Cauchy_testing_2}
	\begin{aligned}
		\begin{multlined}[t] \int^{t'}_0(\aaa w_{ttt}, w_{tt})_{L^2}\ds - c^2  \int^{t'}_0(\bbb \nabla w_{tt}, \nabla w_{t})_{L^2}\ds	+\rt^b \delta  \int^{t'}_0(\frakKtwo \Lconv \nabla w_{ttt}, \nabla w_{tt})_{L^2} \ds\\
			= \tau^a\int^{t'}_0(\frakKone \Lconv \psi^{(\tau)}_{tt}, w_t)_{L^2} \ds - \taua \big(\psitwo , w(0)\big)_{L^2} + \tau^a\int^{t'}_0(\frakKone  \Lconv \nabla\psit^{(\tau)}, \nabla w)_{L^2} \ds.
		\end{multlined}	
	\end{aligned}		
\end{equation}
We intend to handle the right-hand-side terms similarly to before where integration by parts was used:
\begin{equation}\label{eqn:Cauchy2_est1}
	\begin{multlined}
		\Big|\,\tau^a\int^{t'}_0(\frakKone \Lconv \psi^{(\tau)}_{tt}, w_t)_{L^2} \ds - \taua \big(\psitwo , w(0)\big)_{L^2}\,\Big|\\
		\leq  \taua \|\frakKone\|_{\mm(0,t')}\Big(\|\psi_t^{(\tau)}\|_{L^2(0,t';L^2(\Om))}\|w_{tt}\|_{L^2(0,t';L^2(\Om))} \\+ \|\psi_t^{(\tau)}\|_{L^\infty(0,t';L^2(\Om))}
		\|w_t\|_{L^\infty(0,t';L^2(\Om))}  + \|\psitwo\|_{L^2(\Om)} \|w(0)\|_{L^2(\Om)}\Big),
	\end{multlined}
\end{equation}
while the second term was integrated by parts once in time using that the time-reversed variable $\widetilde w (0) = w(t') = 0$ (recall that $\widetilde w(s) = w (t'-s)$ for all $s\in(0,t')$):
\begin{equation}\label{eqn:Cauchy2_est2}
	\begin{multlined}
		\Big|\,\taua \int^{t'}_0(\frakKone \Lconv \nabla\psit^{(\tau)}, \nabla w)_{L^2} \ds\,\Big| 
		\\ \leq \taua \|\frakKone\|_{\mm(0,t')} \Big(\|\nabla \widetilde  w_t\|_{L^2(0,t';L^2(\Om))}\|\nabla \psi^{(\tau)}\|_{L^2(0,t';L^2(\Om))} \\+ \|\nabla \widetilde w \|_{L^\infty(0,t';L^2(\Om))}\|\nabla \psi_0\|_{L^2(\Om)}\Big).
	\end{multlined}
\end{equation}
In contrast to Proposition~\ref{Prop:Cauchy_seq1}, now all the right-hand-side terms with $w$ can be absorbed. Indeed, notice that for $s \in (0,t')$
\begin{align}
w_t(s) = - \int_s^{t'} w_{tt} (\zeta) \textup{d}\, \zeta \quad\textrm{such that}\quad  \|w_t\|_{L^\infty(0,t';L^2(\Om))}  \lesssim_T
 \|w_{tt}\|_{L^2(0,t';L^2(\Om))}.
\end{align}
Similarly, we have that
\[
 \|w(0)\|_{L^2(\Om)}  \lesssim_T
\|w_{tt}\|_{L^2(0,t';L^2(\Om))},
\]
and 
\[\|\nabla \widetilde w \|_{L^\infty(0,t';L^2(\Om))} \lesssim_T \|\nabla \widetilde w_t \|_{L^2(0,t';L^2(\Om))}.
\]

Using Young's inequality in \eqref{eqn:Cauchy2_est1} and \eqref{eqn:Cauchy2_est2}, and piecing the estimates together, we obtain that:
\begin{equation}
	\begin{multlined}
			\aaa \| w_{tt}(t')\|^2_{L^2(\Om)} + c^2 \bbb \|\nabla \widetilde  w_t (t')\|^2_{L^2(\Om)} \\\lesssim_T \tau^{2a} \Check C_{\frakKone, \psi^{(\tau)}, \psi^{(0)}} + \|w_{tt}\|_{L^2(0,t';L^2(\Om))}^2 + \|\nabla \widetilde w_t\|^2_{L^2(0,t';L^2(\Om))},
			\end{multlined}
\end{equation}
where again, $\Check C_{\frakKone, \psi^{(\tau)}, \psi^{(0)}}$ is uniformly bounded in $\tau$ thanks to Theorem~\ref{Prop:Wellposedness_a_g_b},  Proposition~\ref{Prop:wellposedness_GFE_laws}, or, alternatively, Proposition~\ref{Cor:wellposedness_GFEII}.
Gr\"onwall's inequality yields then the desired result.
\qed
\end{proof}
This last theorem allows us to draw conclusions over the limiting behavior of a large class of higher-order linear models of acoustics where it is usually the case that $\aaa=\bbb=1$. We summarize some of the implications for well established fractional Moore--Gibson--Thompson equations in the following corollary.
  
\begin{corollary}\label{Cor:rate_of_MGT}
The rate of convergence:
\[\|1\Lconv\psi^{(\tau)}- 1\Lconv\psi^{(0)}\|_{L^\infty(0,t;L^2(\Omega))} = O(\tau^a) \qquad \textrm{as } \tau \searrow 0.\]
 holds for all equations in Table~\ref{table:r_kernels} with $\aaa = \bbb =1$ (with the restriction $\alpha>\frac12$ for fMGT I and fMGT).
 \\ 
 \noindent Additionally, when supplemented with initial data: \[ (\psi,\psit,\psitt) \Big|_{t=0} = (\psi_0,0,0),\] 
 with $\psi_0 \in H^1_0(\Omega)$, this rate of convergence also holds for the linear fractional MGT equations of the form
 \begin{equation}
 	\tau^a\frakKone \Lconv \psi_{ttt} + \psi_{tt} - c^2  \tau^a\frakKone \Lconv \Delta\psi_t - c^2  \Delta \psi -   \delta \rt^b \frakKtwo \Lconv \Delta\psi_{tt}   = f,
 	\end{equation}
 derived in~\cite{kaltenbacher2022time}. The kernels $\frakKone$ and $\frakKtwo$ are $L^1$-regular and are given in Table~\ref{table:r_kernels}, with the restriction $\alpha>\frac12$ for fMGT I and fMGT.  
\end{corollary} 
For the integer-order MGT equation, Corollary~\ref{Cor:rate_of_MGT} complements the result of~\cite[Theorem 2.4]{bongarti2020singular}.
Indeed, in the aforementioned reference, a rate of convergence is established in the energy norm for initial data in $H^2(\Omega) \times H^2(\Om) \times H^1(\Om)$. Here, we were able to provide convergence rates in weaker norms for initial data in $H^1(\Omega) \times \{0\} \times L^2(\Om)$. As mentioned in the discussion on page~\pageref{discussion:psit0}, the requirement $\psi_1 = 0$ can be dropped in a straightforward manner if $\frakKone = \delta_0$ and we expect the results here to hold for the integer-order MGT equation with initial data in $H^1(\Omega) \times H^1(\Om) \times L^2(\Om)$.

\section*{Conclusion}
In this work, we have shown the flexibility of our framework in studying well-posedness for a large family of higher-order-in-time wave equations. In particular, we have established $\tau$-uniform well-posedness for generalized fractional MGT equations. We have thereafter connected the fMGT equations to second-order-in-time models through a tailored limiting procedure.

Additionally, through a carefully designed testing strategy, we have established the rate of convergence at the vanishing relaxation time limit ($\tau \searrow 0$) in a relatively low-regularity setting. 
As a byproduct, we also obtained novel convergence rates for the MGT equation, thus complementing on the results available in the literature. The low regularity assumed of the initial data allows us to consider extending the results to mixed Neumann/Dirichlet-absorbing boundary conditions, which are relevant for the simulation of acoustic phenomena. 

The flexibilty of our framework allows us also to contemplate generalizing the available decay rate results for the Moore--Gibson--Thompson equation~(see, e.g., \cite{kaltenbacher2011wellposedness,pellicer2019wellposedness}) to the nonlocal setting, thus helping us further understand the long-term behavior of fractional wave equations.
\begin{appendices}
	\section{Additional notes on fractional derivative spaces}
An analogous result to Lemma~\ref{Lemma:Caputo_seq_compact} can be formulated for sequential compactness with Riemann--Liouville-type derivatives. Although not needed in this work, the interested reader might find it useful. It is in particular noteworthy that the completeness of the related space (denoted $Y_{\mathfrak{K}}^p$ below) does not rest on the regularity of the resolvent, in contrast to the space constructed with Caputo--Dzhrabashyan-type derivatives, $X^p_\genk$. 

\begin{lemma}[Compactness of $Y_{\mathfrak{K}}^p$] \label{Lemma:RL_seq_compact}
	Let $1\leq p\leq \infty$ and let $\genk \in \mm(0,T)$. Consider the space
	\[
	Y_{\mathfrak{K}}^p = \{u \in L^p(0,T) \ |\ (\mathfrak{K} \Lconv u)_t \in L^p(0,T)\}
	\]
	with the norm
	$\|u\|_{Y_\genk^p(0,T)} = \big(\|u\|_{L^p}^p + \|(\genk\Lconv u_t)\|_{L^p}^p \big)^{1/p}$, and the usual modification for $p=\infty$.
	
	Then $Y_\genk^p(0,T)$ is reflexive for $1<p<\infty$ and separable for $1 \leq p<\infty$.
	Furthermore, the unit ball of $Y_\genk^p(0,T)$, $B_{Y_\genk}^p$, is weakly sequentially compact for $1<p<\infty$. $B_{Y_\genk}^\infty$ is weak-$*$ sequentially compact. 
\end{lemma}

\begin{proof}
	To show that $Y_{\mathfrak{K}}^p$ is complete, take a Cauchy sequence $(u_n)_{n \geq 1}\subset Y_{\mathfrak{K}}^p$. Then $(u_n)_{n \geq 1}$ and $\big((\genk\Lconv u_n)_t\big)_{n \geq 1}$ are Cauchy sequences in $L^p(0,T)$ and therefore converge to some limits $u$ and $g$, respectively, in $L^p(0,T)$. 
	Since 
	\begin{align}
		\operatorname{T}_{\mathfrak{K}}\ : L^p(0,T) &\rightarrow L^p(0,T) \\
		u&\mapsto \mathfrak{K}\Lconv u,
	\end{align} 
	is bounded (due to \cite[Theorem 3.6.1]{gripenberg1990volterra}) and thus continuous, we have
	\[\mathfrak{K}\Lconv u_n \to \mathfrak{K}\Lconv u \quad \textrm{strongly in }L^p(0,T).\]
	Let $\phi \in C_c^1([0,T])$ and let $n\geq 1$. Then 
	\[
	\int_0^T (\mathfrak{K}\Lconv u_n)_t \,\phi \ds = - \int_0^T (\mathfrak{K}\Lconv u_n) \,\phi_t\ds. 
	\]
	Passing to the limit, we get
	\[
	\int_0^T g \,\phi \ds = - \int_0^T \mathfrak{K}\Lconv u \,\phi_t\ds. 
	\]
	Thus, $u \in Y^p_{\mathfrak{K}}$, $(\mathfrak{K}\Lconv u)_t = g$, and $\|u_n-u\|_{Y_{\mathfrak{K}}^p} \to 0$ as $n\to \infty$.	
	The proofs of compactness are similar to those given in Lemma~\ref{Lemma:Caputo_seq_compact}, so we omit the details here.
	\qed
\end{proof}
Compared to Lemma~\ref{Lemma:Caputo_seq_compact}, the weaker assumption on the existence and regularity of a resolvent, translates however into a failure to obtain an embedding of $Y_\genk^p$ into $C[0,T]$. This embedding can be recovered if a resolvent $\tilde \genk \in L^{p'}(0,T)$ exists and if the initial condition is given by $(\genk\Lconv u) (0) = 0$.
\end{appendices}

\begin{acknowledgements}
	The author is very grateful to Vanja Nikoli\'c (Radboud University) for valuable discussions and comments on the draft.
\end{acknowledgements}

\bibliography{references}{}
\bibliographystyle{spmpsci} 

\end{document}